\documentclass[10pt]{amsart}

\usepackage{float}
\usepackage{amscd,amsmath,amssymb,amsfonts,amsthm,ascmac, bm}
\usepackage{enumerate,stmaryrd,latexsym, comment, mathtools, mathdots} 

\usepackage{varwidth}

\usepackage{amsthm, amssymb}
\usepackage{bm}
\usepackage[leqno]{amsmath}
\usepackage{caption}
\usepackage{subcaption}
\usepackage{hyperref}
\usepackage{relsize}
\usepackage{booktabs}
\usepackage{a4wide}

\usepackage[mathscr]{euscript}

\usepackage[all]{xy}

\usepackage{ragged2e}

\usepackage{tikz-cd}
\usepackage{tikz}
\usetikzlibrary{ 
	3d, matrix, decorations.pathreplacing,calc,decorations.pathmorphing, fit, patterns, trees, decorations.markings, arrows, positioning, shapes}

\tikzstyle{Dnode}=[draw, circle, inner sep = 0.07cm]
\tikzstyle{double line} = [
decoration={
    markings,
    mark=at position 0.55 with {\arrow[line width = 0.5pt,scale=1]{angle 90}}},
	double distance = 1.5pt, 
	double=\pgfkeysvalueof{/tikz/commutative diagrams/background color},
	postaction={decorate}
]

\tikzstyle{triple line} = [
decoration={
    markings,
    mark=at position 0.55 with {\arrow[line width = 0.5pt,scale=1]{angle 90}}},
	double distance = 2pt, 
	double=\pgfkeysvalueof{/tikz/commutative diagrams/background color},
	postaction={decorate}
]

\numberwithin{equation}{section}
\allowdisplaybreaks[1]

\usepackage{colonequals} 

\allowdisplaybreaks

\DeclareMathOperator{\Cone}{Cone}
\DeclareMathOperator{\PC}{PC}
\DeclareMathOperator{\Cox}{Cox}

\DeclareMathOperator{\Conv}{Conv}

\DeclareMathOperator{\indeg}{deg^{--}}

\newcommand{\defi}[1]{{\textit{#1}}}

\newcommand{\C}{{\mathbb{C}}}

\newcommand{\R}{{\mathbb{R}}}

\newcommand{\Z}{{\mathbb{Z}}}

\DeclareMathOperator{\flag}{Fl}

\def\Sn{\mathfrak{S}_{n}}

\newcommand{\Q}{{\mathsf{Q}}}

\newcommand{\eset}[2]{E_{#1}^{-}(#2)}
\newcommand{\peset}[2]{E_{#1}^{+}(#2)}

\def\G{\mathcal{G}}

\def\GL{\rm GL}

\theoremstyle{plain}
\newtheorem{theorem}{Theorem}[section]
\newtheorem{lemma}[theorem]{Lemma}
\newtheorem{proposition}[theorem]{Proposition}
\newtheorem{corollary}[theorem]{Corollary}

\theoremstyle{definition}
\newtheorem{example}[theorem]{Example}
\newtheorem{definition}[theorem]{Definition}

\theoremstyle{remark}
\newtheorem{remark}[theorem]{Remark}

\hypersetup{colorlinks}
\hypersetup{
	unicode=false,          
	pdftoolbar=true,        
	pdfmenubar=true,        
	pdffitwindow=false,     
	pdfstartview={FitH},    
	pdftitle={My title},    
	pdfauthor={Author},     
	pdfsubject={Subject},   
	pdfcreator={Creator},   
	pdfproducer={Producer}, 
	pdfkeywords={keyword1} {key2} {key3}, 
	pdfnewwindow=true,      
	colorlinks=true,       
	linkcolor=blue,          
	citecolor=red,        
	filecolor=violet,      
	urlcolor=violet       
}

\tikzset{root/.style = {circle, double, draw, inner sep = 1pt}}
\tikzset{vertex/.style = {circle, fill, inner sep = 1.5pt}}

\makeatletter
\def\namedlabel#1#2{\begingroup
   \def\@currentlabel{#2}%
   \label{#1}\endgroup
}
\makeatother

\newcommand{\dyn}[1]{{\mathsf{#1}}}

\makeatletter
\renewcommand\theenumii{(\alph{enumii})}

\renewcommand\p@enumiii{\theenumi\theenumii}
\makeatother
\begin{document}

\author[Eunjeong Lee]{Eunjeong Lee}
\address[E. Lee]{Department of Mathematics,
	Chungbuk National University,
	Cheongju 28644, Republic of Korea}
\email{eunjeong.lee@chungbuk.ac.kr}

\author[Mikiya Masuda]{Mikiya Masuda}
\address[M. Masuda]{Osaka Central Advanced Mathematics Institute (OCAMI), Osaka Metropolitan  
University, Sumiyoshi-ku, Sugimoto, 558-8585, Osaka, Japan}
\email{mikiyamsd@gmail.com}

\author[Seonjeong Park]{Seonjeong Park}
\address[S. Park]{Department of Mathematics Education, Jeonju University, Jeonju 55069, Republic of Korea}
\email{seonjeongpark@jj.ac.kr}

\thanks{
Lee was supported by the National Research Foundation of Korea(NRF) grant funded by the Korea government(MSIT) (No.\ RS-2022-00165641, No.\ RS-2023-00239947).
Masuda was supported in part by 
the HSE University Basic Research Program. Park was supported by the Basic Science Research Program through the National Research Foundation of Korea (NRF) funded by the Government of Korea (NRF-2020R1A2C1A01011045). This work was partly supported by Osaka City University Advanced Mathematical Institute (MEXT Joint Usage/Research Center on Mathematics and Theoretical Physics JPMXP0619217849).}

\title{Toric Schubert varieties and directed Dynkin diagrams}

\date{\today}



\begin{abstract} 
A flag variety is a homogenous variety $G/B$ where $G$ is a simple algebraic group over the complex numbers and $B$ is a Boel subgroup of $G$. A Schubert variety $X_w$ is a subvariety of $G/B$ indexed by an element $w$ in the Weyl group of $G$. It is called toric if it is a toric variety with respect to the maximal torus of $G$ in $B$. In this paper, we associate an edge-labeled digraph~$\mathcal{G}_w$ with a toric Schubert variety $X_w$ and classify toric Schubert varieties up to isomorphism. We also give a simple criterion of when a toric Schubert variety $X_w$ is (weak) Fano in terms of~$\G_w$. Finally, we discuss whether toric Schubert varieties can be distinguished by their integral cohomology rings up to isomorphism and show that this is the case when $G$ is of simply-laced type. 
\end{abstract}

\maketitle

\setcounter{tocdepth}{1} 
\tableofcontents

\section{Introduction} 

A flag variety is a homogeneous variety $G/B$ where $G$ is a simple algebraic group over the complex numbers $\C$ and $B$ is a Borel subgroup of $G$. Let $T$ be the maximal torus of $G$ in $B$ and $W$ the Weyl group of $G$. The Bruhat decomposition $G/B=\bigsqcup_{w\in W}BwB/B$ provides a cell decomposition of $G/B$. Indeed, $BwB/B$ is isomorphic to an affine space $\C^{\ell(w)}$ where $\ell(w)$ denotes the length of $w$. The Schubert variety $X_w$ is the closure of the cell $BwB/B$ in the flag variety~$G/B$, and $X_w$ is called toric if it is a toric variety with respect to the torus $T$. It is known that $X_w$ is toric if and only if $w$ is a product of \emph{distinct} simple reflections (see \cite{Fan98, Karu13Schubert}). This implies that a toric Schubert variety $X_w$ is smooth, in fact, $X_w$ is a Bott manifold that is the total space of an iterated $\C P^1$-bundle over a point. 

In this paper, we associate an edge-labeled digraph $\mathcal{G}_w$ with a toric Schubert variety $X_w$ and prove the following. 

\begin{theorem}[Theorem~\ref{theo:graph_determines_Xw}] \label{theo:main1}
Let $W$ and $W'$ be the Weyl groups of simple algebraic groups~$G$ and~$G'$, respectively. Then toric Schubert varieties $X_w$ $(w\in W)$ and $X_{w'}$ $(w'\in W')$ are isomorphic as varieties if and only if $\G_{w}$ and $\G_{w'}$ are isomorphic as edge-labeled digraphs. 
\end{theorem}

The edge-labeled digraph $\mathcal{G}_w$ is defined as follows. Take a reduced expression $s_{i_1}s_{i_2}\cdots s_{i_m}$ of~$w$ by simple reflections $s_i$'s in $W$, where $i_1,i_2,\dots,i_m$ are mutually distinct. Then the vertex set of~$\mathcal{G}_w$ is $\{i_1,\dots,i_m\}$, two vertices in $\mathcal{G}_w$ are joined by a directed edge $(i_j,i_k)$ if and only if $j>k$ and the $(i_j,i_k)$ entry $c_{i_j,i_k}$ in the Cartan matrix of Lie type $G$ is nonzero, and we label the directed edge by assigning $-c_{i_j,i_k}$ which is $1$, $2$, or $3$. Therefore, the underlying graph of $\mathcal{G}_w$ is an induced subgraph of the Dynkin diagram of Lie type $G$ with a multiple edge (if any) replaced by a single edge. The labels on the directed edges in $\mathcal{G}_w$ are $1$ in most cases. The only exception is the directed edge in $\mathcal{G}_w$ corresponding to the multiple edge in the Dynkin diagram and having the direction of the multiple edge. In this case, the label on the directed edge is $2$ if the multiple edge is a double edge and $3$ if it is a triple edge. 

As the next theorem shows, the edge-labeled digraph $\mathcal{G}_w$ is also useful for determining whether the toric Schubert variety $X_w$ is Fano or weak Fano. 

\begin{theorem}[Theorem~\ref{thm_Fano_weak_Fano}] \label{theo:main2}
A toric Schubert variety $X_w$ is Fano {\rm(}resp. weak Fano{\rm)} if and only if every vertex of $\mathcal{G}_w$ has indegree at most $1$ {\rm(}resp. $2${\rm)}.
\end{theorem} 

An element of the Weyl group $W$ is called a Coxeter element if it is a product of all simple reflections in $W$. When $w\in W$ is a Coxeter element, the underlying graph of $\mathcal{G}_w$ agrees with the Dynkin diagram of Lie type $G$ with a multiple edge (if any) replaced by a single edge. We enumerate the number of isomorphism classes of (Fano or weak Fano) toric Schubert varieties $X_w$ for Coxeter elements $w$ in each Lie type (see Table~\ref{table_cardinality_iso_classes_Coxeter} in Section~\ref{sect:enumeration}). 

We also discuss whether a toric Schubert variety $X_w$ is determined by its cohomology ring $H^*(X_w;\Z)$, in other words, whether $\G_w$ can be recovered from $H^*(X_w;\Z)$. 

\begin{theorem}[Theorem~\ref{theo:recover}] \label{theo:main3}
The edge-labeled digraph $\mathcal{G}_w$ can be recovered {\rm(}up to isomorphism{\rm)} from the cohomology ring $H^*(X_w;\Z)$ if 
all the labels in $\G_w$ are $1$. 
\end{theorem}

Combining Theorem~\ref{theo:main3} with Theorem~\ref{theo:main1}, we obtain the following corollary. 

\begin{corollary}
Let $W$ and $W'$ be the Weyl groups of simple algebraic groups $G$ and $G'$ of simply-laced type {\rm(}i.e. type $\dyn{A}$, $\dyn{D}$, or $\dyn{E}${\rm)}. Then toric Schubert varieties $X_w$ $(w\in W)$ and $X_{w'}$ $(w'\in W')$ are isomorphic as varieties if and only if $H^{\ast}(X_w;\Z)$ and $H^{\ast}(X_{w'};\Z)$ are isomorphic as graded rings.
\end{corollary}

\begin{remark} \label{rema:1}
	The assumption in Theorem~\ref{theo:main3} cannot be weakened as is seen in the following examples. 
	\begin{enumerate}
		\item 
		When $G$ is of type $\dyn{G}_2$, there are two Coxeter elements $w,w'$ and both $\mathcal{G}_w$ and $\mathcal{G}_{w'}$ consist of two vertices with a directed edge but one has label $1$ while the other has label $3$. However, the cohomology rings $H^*(X_w;\Z)$ and $H^*(X_{w'};\Z)$ are isomorphic. Indeed, if $F_a$ denotes the Hirzebruch surface indexed by a nonnegative integer $a$, then $X_w$ and $X_{w'}$ are $F_1$ and $F_3$. As is well-known, $F_a$ and $F_{a'}$ are isomorphic if and only if $a=a'$, and $H^*(F_a;\Z)$ and $H^*(F_{a'};\Z)$ are isomorphic as graded rings (more strongly, $F_a$ and $F_{a'}$ are diffeomorphic) if and only if $a\equiv a'\pmod{2}$. 
		
		\item When $G$ is of type $\dyn{C}_3$, $\G_{s_1s_3}$ consists of two vertices with no edge and $\G_{s_2s_3}$ consists of two vertices and a directed edge with label $2$. On the other hand, $X_{s_1s_3}$ and $X_{s_2s_3}$ are isomorphic to $F_0$ and $F_2$ respectively, so their integral cohomology rings are isomorphic. 
		We obtain a similar observation when $G$ is of type $\dyn{B}_3$ or $\dyn{F}_4$.
		The toric Schubert variety $X_{s_3s_2}$ in type $\dyn{B}_3$ or $\dyn{F}_4$ is isomorphic to the Hirzebruch surface $F_2$ 
		while $X_{s_1s_3}$ is isomorphic to $F_0$. 
	\end{enumerate}
\end{remark}

As mentioned in Remark~\ref{rema:1}, not all toric Schubert varieties are distinguished as varieties by their integral cohomology rings. However, they are distinguished by their integral cohomology rings as smooth manifolds. Indeed, toric Schubert varieties are Bott manifolds and it is known that any Bott manifolds are distinguished by their integral cohomology rings up to diffeomorphism~(see~\cite{Choi_et}). 
Related to this, it is asked and studied in \cite{richmond2021isomorphism} whether smooth (not necessarily toric) Schubert varieties are distinguished by their integral cohomology rings up to diffeomorphism or homeomorphism. 

This paper is organized as follows. As mentioned above, toric Schubert varieties are Bott manifolds. In Section~\ref{sec:general}, we review fans of Bott manifolds and recall a criterion of when they are Fano or weak Fano. In Section~\ref{section_toric_Schubert_directed_graphs}, we associate the edge-labeled digraph $\mathcal{G}_w$ with a toric Schubert variety $X_w$ and prove Theorems~\ref{theo:main1} and~\ref{theo:main2}. Using Theorem~\ref{theo:main1} we enumerate the isomorphism classes of (Fano or weak Fano) toric Schubert varieties in Section~\ref{sect:enumeration}. In Section~\ref{section_cohomology_ring_ADE} we prove Theorem~\ref{theo:main3} using a presentation of $H^*(X_w;\Z)$ as a graded ring with generators and relations.


\section{Preliminaries: toric Schubert varieties and Bott manifolds}\label{sec:general}
In this section, we recall the classification of toric Schubert varieties and their description. To do so, we review Bott manifolds, which are smooth projective toric varieties. Moreover, we consider Fano and weak Fano conditions on Bott manifolds.

\subsection{Toric Schubert varieties}
Let $G$ be a simple algebraic group over $\C$ of rank $r$, let $B$ be a Borel subgroup, and let $T$ be a maximal torus of $G$ in $B$. We denote by $W$ the Weyl group of $G$. A \emph{flag variety} is the homogeneous space $G/B$, which is a smooth projective variety. 
When $G$ is of type $A$ with rank $n-1$, then $G/B$ is diffeomorphic to 
\[
\flag(n) \colonequals \{(\{0\} \subset V_1 \subset V_2 \subset \cdots \subset V_{n} = \C^{n}) \mid \dim_{\C} V_i = i \text{ for all } i =1,\dots,n\},
\]
where each $V_i$ is a linear subspace of $\C^{n}$. Moreover, the Weyl group is the symmetric group $\Sn$ on the set $[n] \colonequals \{1,\dots,n\}$.

The Weyl group $W$ of $G$ is generated by simple reflections $s_i$ for $i=1,\dots,r$, so each element $w \in W$ can be expressed by a product of generators:
\[
w = s_{i_1} s_{i_2} \cdots s_{i_m}. 
\]
If $m$ is minimal among all such expressions for $w$, then $m$ is called the \emph{length} of $w$ and we write $\ell(w) = m$. Moreover, we call the word $s_{i_1} s_{i_2} \cdots s_{i_m}$ a \emph{reduced decomposition} for $w$. A decomposition $s_{i_1} s_{i_2} \cdots s_{i_m}$ provides a string $(i_1,i_2,\dots,i_m)$ in $[r]^k$ and we call it a \emph{word}. A word $(i_1,i_2,\dots,i_m)$ is \emph{reduced} if the corresponding decomposition $s_{i_1} s_{i_2} \cdots s_{i_m}$ is reduced. 
In Table~\ref{table_finite}, we provide Dynkin diagrams for finite types. 
In this manuscript, we use the ordering on the simple roots as in the table following~\cite{Humphreys78Lie}.

\begin{table}[t]
\begin{center}
\begin{tabular}{c|l  }
\toprule
$\Phi$ & Dynkin diagram \\
\midrule
$\dyn{A}_r$ $(r \geq 1)$ &
\begin{tikzpicture}[scale=.5, baseline=-.5ex]
	\tikzset{every node/.style={scale=0.7}}
	\node[Dnode] (1) {} node[below = 0cm of 1] {$1$};
	\node[Dnode] (2) [right = of 1] {} node[below=0cm of 2] {$2$};
	\node[Dnode] (3) [right = of 2] {} node[below=0cm of 3] {$3$};
	\node[Dnode] (4) [right =of 3] {} node[below=0cm of 4] {$r-1$};
	\node[Dnode] (5) [right =of 4] {} node[below=0cm of 5] {$r$};			

	\draw (1)--(2)--(3)
		(4)--(5);
	\draw[dotted] (3)--(4);

\end{tikzpicture}  \\[1em] 			
$\dyn{B}_r$ $(r \geq 2)$ &
\begin{tikzpicture}[scale=.5, baseline=-.5ex]
	\tikzset{every node/.style={scale=0.7}}
			
	\node[Dnode] (1) {} node[below = 0cm of 1] {$1$};
	\node[Dnode] (2) [right = of 1] {} node[below = 0cm of 2] {$2$};
	\node[Dnode] (3) [right = of 2] {} node[below = 0cm of 3] {$r-2$};
	\node[Dnode] (4) [right =of 3] {} node[below = 0cm of 4] {$r-1$};
	\node[Dnode] (5) [right =of 4] {} node[below = 0cm of 5] {$r$};

	\draw (1)--(2)
		(3)--(4);
	\draw [dotted] (2)--(3);
	\draw[double line] (4)--(5);
\end{tikzpicture}  \\[1em]		
$\dyn{C}_r$ $(r \geq 3)$ & 
\begin{tikzpicture}[scale=.5, baseline=-.5ex]
	\tikzset{every node/.style={scale=0.7}}
			
	\node[Dnode] (1) {} node[below = 0cm of 1] {$1$};
	\node[Dnode] (2) [right = of 1] {} node[below = 0cm of 2] {$2$};
	\node[Dnode] (3) [right = of 2] {} node[below = 0cm of 3] {$r-2$};
	\node[Dnode] (4) [right =of 3] {} node[below = 0cm of 4] {$r-1$};
	\node[Dnode] (5) [right =of 4] {} node[below = 0cm of 5] {$r$};

	\draw (1)--(2)
		(3)--(4);
	\draw [dotted] (2)--(3);
	\draw[double line] (5)--(4);
\end{tikzpicture}  \\[1em] 
			
$\dyn{D}_r$ $(r \geq 4)$ & 
\begin{tikzpicture}[scale=.5, baseline=-.5ex]
	\tikzset{every node/.style={scale=0.7}}
			
	\node[Dnode] (1) {} node[below = 0cm of 1] {$1$};
	\node[Dnode] (2) [right = of 1] {} node[below = 0cm of 2] {$2$};
	\node[Dnode] (3) [right = of 2] {} node[below = 0cm of 3] {$r-3$};
	\node[Dnode] (4) [right =of 3] {} node[below = 0cm of 4] {$r-2$};			
			
	\node[Dnode] (5) [above right= 0.3cm and 1cm of 4] {} node[right = 0cm of 5] {$r-1$};
	\node[Dnode] (6) [below right= 0.3cm and 1cm of 4] {} node[right = 0cm of 6] {$r$};

	\draw(1)--(2)
		(3)--(4)--(5)
		(4)--(6);
	\draw[dotted] (2)--(3);
\end{tikzpicture}
\\[1em] 
$\dyn{E}_6$& 
\begin{tikzpicture}[scale=.5, baseline=-.5ex]
	\tikzset{every node/.style={scale=0.7}}

	\node[Dnode] (1) {} ;
	\node[Dnode] (3) [right=of 1] {} ;
	\node[Dnode] (4) [right=of 3] {} ;
	\node[Dnode] (2) [above=of 4] {} node[right = 0cm of 2] {$2$};
	\node[Dnode] (5) [right=of 4] {} ;
	\node[Dnode] (6) [right=of 5] {} ;
	
	\foreach \x in {1,3,4,5,6}{
		\node[below = 0cm of \x] {$\x$};
	}

	\draw(1)--(3)--(4)--(5)--(6)
		(2)--(4);
\end{tikzpicture}			
\\[1em]
$\dyn{E}_7$ & 
\begin{tikzpicture}[scale=.5, baseline=-.5ex]
	\tikzset{every node/.style={scale=0.7}}

	\node[Dnode] (1) {} ;
	\node[Dnode] (3) [right=of 1] {} ;
	\node[Dnode] (4) [right=of 3] {} ;
	\node[Dnode] (2) [above=of 4] {} node[right = 0cm of 2] {$2$};
	\node[Dnode] (5) [right=of 4] {} ;
	\node[Dnode] (6) [right=of 5] {} ;
	\node[Dnode] (7) [right=of 6] {} ;

	\foreach \x in {1,3,4,5,6,7}{
		\node[below = 0cm of \x] {$\x$};
	}

	\draw(1)--(3)--(4)--(5)--(6)--(7)
		(2)--(4);
\end{tikzpicture}	
\\[1em]
$\dyn{E}_8$ & 
\begin{tikzpicture}[scale=.5, baseline=-.5ex]
	\tikzset{every node/.style={scale=0.7}}

	\node[Dnode] (1) {} ;
	\node[Dnode] (3) [right=of 1] {} ;
	\node[Dnode] (4) [right=of 3] {} ;
	\node[Dnode] (2) [above=of 4] {} node[right = 0cm of 2] {$2$};
	\node[Dnode] (5) [right=of 4] {} ;
	\node[Dnode] (6) [right=of 5] {} ;
	\node[Dnode] (7) [right=of 6] {};
	\node[Dnode] (8) [right=of 7] {};

	\foreach \x in {1,3,4,5,6,7,8}{
		\node[below = 0cm of \x] {$\x$};
	}

	\draw(1)--(3)--(4)--(5)--(6)--(7)--(8)
		(2)--(4);
\end{tikzpicture}
\\[1em]
$\dyn{F}_4$
&
\begin{tikzpicture}[scale = .5, baseline=-.5ex]
	\tikzset{every node/.style={scale=0.7}}
			
	\node[Dnode] (1) {};
	\node[Dnode] (2) [right = of 1] {};
	\node[Dnode] (3) [right = of 2] {};
	\node[Dnode] (4) [right =of 3] {};

	\foreach \x in {1,...,4}{
		\node[below = 0cm of \x] {$\x$};
	}
			
	\draw (1)--(2)
		(3)--(4);
	\draw[double line] (2)-- (3);
\end{tikzpicture}  \\[1em]
$\dyn{G}_2$
&
\begin{tikzpicture}[scale =.5, baseline=-.5ex]
	\tikzset{every node/.style={scale=0.7}}
			
	\node[Dnode] (1) {};
	\node[Dnode] (2) [right = of 1] {};

	\foreach \x in {1,2}{
		\node[below = 0cm of \x] {$\x$};
	}
			
	\draw[triple line] (2)--(1);
	\draw (1)--(2);
			
\end{tikzpicture}\\
\bottomrule
\end{tabular}
\end{center}
\caption{Dynkin diagrams for finite types}\label{table_finite}
\end{table}

The left multiplication of $T$ on $G$ induces an action of $T$ on $G/B$. Then there is a bijective correspondence between the $T$-fixed point set $(G/B)^T$ and the Weyl group $W$ of $G$, and we denote by $wB$ the fixed point set in $G/B$ corresponding to $w \in W$. 
For $w \in W$, the \emph{Schubert variety} $X_w$ is a subvariety of $G/B$ defined by the (Zariski) closure of $BwB/B \subset G/B$. A Schubert variety is invariant under the $T$-action on $G/B$ of complex dimension $\ell(w)$.

Considering a reduced decomposition for $w \in W$, one can decide whether the Schubert variety~$X_w$ is toric or not with respect to the torus action $T$ as follows.
\begin{theorem}[{\cite{Fan98, Karu13Schubert}}]\label{thm_toric_Schubert_distinct}
For $w \in W$, the following statements are equivalent:
\begin{enumerate}
\item $X_w$ is a toric variety.
\item $X_w$ is a smooth toric variety.
\item A reduced decomposition for $w$ consists of distinct letters.
\end{enumerate}
\end{theorem}

\begin{example}\label{example_toric_A2}
There are six elements in the Weyl group $\mathfrak{S}_3$ of type $\dyn{A}_2$. For each $w \in \mathfrak{S}_3$, we display whether the corresponding Schubert variety $X_w$ is toric or not and the length $\ell(w) = \dim_{\C} X_w$ in the following table. 
\begin{center}
\begin{tabular}{c|cccccc}
\toprule
$w$  & $e$ & $s_1$ & $s_2$ & $s_1s_2$ & $s_2s_1$ & $s_1s_2s_1$ \\
\midrule 
$X_w$ is toric &  yes & yes & yes & yes & yes & no \\
$\ell(w)$ & $0$ & $1$ & $1$ & $2$ & $2$ & $3$ \\
\bottomrule
\end{tabular}
\end{center}
\end{example}

\begin{example}\label{example_toric_B2}
The Weyl group $W_{\dyn{B}_2}$ of type $\dyn{B}_2$ is given by $W_{\dyn{B}_2} = \langle s_1, s_2 \mid (s_1)^2 = (s_2)^2 = (s_1s_2)^4 = e\rangle$. There are eight elements in $W_{\dyn{B}_2}$. Five of them produce toric Schubert varieties. 
\begin{center}
\begin{tabular}{c|cccccccc}
\toprule
$w$ & $e$ & $s_1$ & $s_2$ & $s_1s_2$ & $s_2s_1$ & 
	$s_1s_2s_1$ & $s_2s_1s_2$ & $s_1s_2s_1s_2$ \\
\midrule 
$X_w$ is toric & yes & yes & yes & yes & yes & no & no & no  \\
$\ell(w)$ & $0$ & $1$ & $1$ & $2$ & $2$ & $3$ & $3$ & $4$ \\
\bottomrule
\end{tabular}
\end{center}
\end{example}

When $G$ is of type $\dyn{A}_{n-1}$, the fan of a toric Schubert variety $X_w$ is the same as the normal fan of a polytope 
\[
\Q_{w} \colonequals \Conv\{(v^{-1}(1),\dots,v^{-1}(n)) \in \R^n \mid v \leq w\}. 
\]
Here, we compare two elements $v,w \in W$ with respect to the \emph{Bruhat order}, that is, for $v = s_{i_1} \cdots s_{i_{k}}$, $u \leq v$ if and only if 
there exists a reduced decomposition $u = s_{i_{j_1}} s_{i_{j_2}} \cdots s_{i_{j_q}}$ with $1 \leq j_1 < \dots < j_q \leq k$. To consider the fan of a toric Schubert variety $X_w$ in general Lie types, we recall~\cite[\S 3.7]{GK94Bott} and \cite[\S 4.3]{LMP_torus_orbit_closures_book_chapter}.  

\begin{theorem}[{\cite[Theorem~4.23]{LMP_torus_orbit_closures_book_chapter}}]\label{thm_char_matrix_of_toric_Schubert}
Let $w = s_{i_1} \cdots s_{i_m}$ be a reduced decomposition for $w \in W$. Assume that $i_1,\dots,i_m$ are distinct. Then the fan of the toric Schubert variety $X_w$ is isomorphic to the fan in $\R^m$ such that primitive ray vectors are the $2m$ column vectors of the following matrix 
and a subset of the column vectors forms a cone if and only if it does not contain both the $i$th column vectors in the left $m \times m$ submatrix and the right $m \times m$ submatrix for each $i=1,\dots,m$:
\begin{equation}\label{equation_ray_vectors}
\left[
\begin{array}{cccc|cccc}
1 & 0 & \cdots & 0 & -1 & 0 & \cdots & 0 \\
0 & 1 &  \cdots & 0 & & -1 & \cdots & 0 \\
\vdots & & \ddots & & & -c_{i_j,i_k} & \ddots & \\
0 & 0 & \cdots & 1 & & & & -1
\end{array}
\right],
\end{equation}
where the $(j,k)$ entry for $m \geq j >k \geq 1$ in the right submatrix above is $-c_{i_j,i_k}$ and $c_{i,j}$ are Cartan integers. Therefore, if we denote the $2m$ ray vectors in the above matrix from left to right by $\mathbf{v}_{i_1},\dots,\mathbf{v}_{i_m},\mathbf{w}_{i_1},\dots,\mathbf{w}_{i_m}$, then 
\begin{equation} \label{equation_ray_vectors2}
\mathbf{v}_{i_k}+\mathbf{w}_{i_k}=\sum_{j>k}(-c_{i_j,i_k})\mathbf{v}_{i_j}
\end{equation}
for $k=1,\dots,m$. 
\end{theorem}
We briefly explain how to obtain Theorem~\ref{thm_char_matrix_of_toric_Schubert}. 
Suppose that $w = s_{i_1} \cdots s_{i_m}$ be a reduced decomposition for $w \in W$ consisting of distinct letters. 
Then the Schubert variety $X_w$ is isomorphic to the \emph{Bott--Samelson variety} corresponding to the decomposition $(i_1,..,i_m)$. Here, a Bott--Samelson variety is a smooth projective variety that can be understood as the total space of an iterated $\C P^1$-bundle whose bundle structure is decided by the decomposition $(i_1,\dots,i_m)$. Moreover, such a Bott--Samelson variety is also toric and its fan structure is described in the paper by Grossberg and Karshon (see~\cite[\S 3.7]{GK94Bott}).

We call the right $m \times m$ submatrix in~\eqref{equation_ray_vectors} the \emph{reduced characteristic matrix}. 
Recall that the Cartan integers $c_{i,j}$ can be read directly from the Dynkin diagram as follows:
\begin{center}
\setlength{\tabcolsep}{20pt}
\begin{tabular}{ccc}
\begin{tikzpicture}[scale =.5, baseline=-.5ex]
\tikzset{every node/.style={scale=0.7}}

\node[Dnode, label=below:{$i$}] (2) {};
\node[Dnode, label=below:{$j$}] (3) [right=of 2] {};

\draw (2)-- (3);

\end{tikzpicture}&
\begin{tikzpicture}[scale =.5, baseline=-.5ex]
\tikzset{every node/.style={scale=0.7}}

\node[Dnode, label=below:{$i$}] (2) {};
\node[Dnode, label=below:{$j$}] (3) [right=of 2] {};

\draw[double line] (2)--(3);

\end{tikzpicture}&
\begin{tikzpicture}[scale =.5, baseline=-.5ex]
\tikzset{every node/.style={scale=0.7}}

\node[Dnode,  label=below:{$i$}] (2) {};
\node[Dnode, label=below:{$j$}] (3) [right=of 2] {};

\draw[triple line] (2)--(3);
\draw (2)--(3);

\end{tikzpicture}
\\
$c_{i,j} = -1$
&$c_{i,j} = -2$
& $c_{i,j} = -3$
\\
$c_{j,i} = -1$
&$c_{j,i} = -1$
& $c_{j,i} = -1$
\end{tabular}
\end{center}
Here, we notice that an `arrow' on the Dynkin diagram represents the lengths of two roots. Indeed, an arrow points to the shorter of the two roots. (Hence, one can also consider an arrow as an inequality comparing the length of two roots.) In the above diagrams, we have $\|\alpha_{i}\| > \|\alpha_{j}\|$ for two simple roots $\alpha_i$ and $\alpha_j$ associated to two vertices $i$ and $j$ in the Dynkin diagram.
\begin{example}\label{example_char_matrix_s31452}
Suppose that $G$ is of type $\dyn{A}_5$. Let $w = s_3s_1s_4s_5s_2$. 
Then the reduced characteristic matrix in~\eqref{equation_ray_vectors} 
is given as follows. 

\[
\begin{bmatrix}
-1 & 0 & 0 & 0 & 0 \\
0 & -1 & 0 & 0 & 0 \\
1 & 0 & -1 & 0 & 0 \\
0 & 0 & 1 & -1 & 0 \\
1 & 1 & 0 & 0 & -1
\end{bmatrix}
\]
Since $(i_1,i_2,i_3,i_4,i_5)=(3,1,4,5,2)$, the column vectors above are $\mathbf{w}_3, \mathbf{w}_1, \mathbf{w}_4, \mathbf{w}_5, \mathbf{w}_2$ from the left.  Therefore we have 
\[
\mathbf{v}_3+\mathbf{w}_3=\mathbf{v}_4+\mathbf{v}_2,\quad \mathbf{v}_1+\mathbf{w}_1=\mathbf{v}_2,\quad \mathbf{v}_4+\mathbf{w}_4=\mathbf{v}_5,\quad \mathbf{v}_5+\mathbf{w}_5=\mathbf{0}, \quad\mathbf{v}_2+\mathbf{w}_2=\mathbf{0}.
\]
\end{example}

\begin{example} 
Suppose that $G$ is of type $\dyn{B}_2$. As computed in Example~\ref{example_toric_B2}, there are five toric Schubert varieties. Considering toric Schubert varieties of dimension $2$, we obtain the reduced characteristic matrices in~\eqref{equation_ray_vectors} as follows:
\begin{center}
\begin{tabular}{cc}
$s_1s_2$ & \qquad\qquad$s_2s_1$ \\
\(
\begin{bmatrix}
-1 & 0 \\ 1 & -1
\end{bmatrix}
\)
& \qquad\qquad\(
\begin{bmatrix}
-1 & 0 \\ 2 & -1
\end{bmatrix}
\)
\end{tabular}
\end{center}
For $w = s_1s_2$, the $(2,1)$-entry of the reduced characteristic matrix is $-c_{i_2,i_1} = -c_{2,1} = -(-1)=1$.
Moreover, for $w = s_2s_1$, the $(2,1)$-entry of the reduced characteristic matrix is $-c_{i_2,i_1} = -c_{1,2} = -(-2)=2$.
Therefore,
\[
\mathbf{v}_1+\mathbf{w}_1=\mathbf{v}_2,\quad \mathbf{v}_2+\mathbf{w}_2=\mathbf{0}
\]
in the former case and 
\[
\mathbf{v}_2+\mathbf{w}_2=2\mathbf{v}_1,\quad \mathbf{v}_1+\mathbf{w}_1=\mathbf{0}
\]
in the latter case. 
\end{example}

\subsection{Fano or weak Fano toric varieties}
There is a combinatorial way to determine whether a smooth compact toric variety is Fano or weak Fano.
For a fan $\Sigma$, a subset $R$ of the primitive ray vectors is called a \defi{primitive collection} of~$\Sigma$ if 
\[
\Cone(R) \notin \Sigma
\quad \text{ but }\quad \Cone(R \setminus \{\mathbf{u}\}) \in \Sigma \quad \text{ for every }\mathbf{u} \in R.
\] 
We denote by $\PC(\Sigma)$ the set of primitive collections of $\Sigma$.
\begin{example}
For the toric Schubert variety $X_w$ in Theorem~\ref{thm_char_matrix_of_toric_Schubert}, the primitive ray vectors of the fan $\Sigma$ of $X_w$ are the column vectors $C:=\{ \mathbf{v}_{i_1},\dots,\mathbf{v}_{i_m},\mathbf{w}_{i_1},\dots,\mathbf{w}_{i_m}\}$ in \eqref{equation_ray_vectors}. 
Then 
\begin{equation}\label{equation_PC_of_Bott}
\PC(\Sigma) = \{ \{\mathbf{v}_{i_k},\mathbf{w}_{i_k}\} \mid k =1,\dots,m\}
\end{equation}
because a subset of $C$ forms a cone if and only if it does not contain both $\mathbf{v}_{i_k}$ and $\mathbf{w}_{i_k}$ for each $k$ as mentioned in Theorem~\ref{thm_char_matrix_of_toric_Schubert}.   
\end{example}

For a primitive collection $R = \{\mathbf{u}'_1, \dots,\mathbf{u}'_{\ell}\}$, we get $\mathbf{u}'_1 + \cdots+\mathbf{u}'_{\ell}=\boldsymbol{0}$ or there exists a unique cone~$\sigma$ of positive dimension such that $\mathbf{u}'_1 + \cdots+\mathbf{u}'_{\ell}$ is in the interior of $\sigma$. That is, 
\begin{equation}\label{eq:primitive}
	\mathbf{u}'_1 + \cdots+\mathbf{u}'_{\ell}=\begin{cases}
	\boldsymbol{0}, &\text{ or }\\
	a_1 \mathbf{u}_1 + \cdots+ a_{s} \mathbf{u}_{s},&{}
	\end{cases}
\end{equation}
where $\mathbf{u}_1,\dots,\mathbf{u}_{s}$ are the primitive generators of $\sigma$ and $a_1,\dots,a_{s}$ are positive integers.
We call~\eqref{eq:primitive} a \emph{primitive relation}, and the \emph{degree} $\deg R$ of a primitive collection $R$ is defined to be 
\begin{equation} \label{eq:degree}
\deg R  \colonequals 
\begin{cases}
\ell &  \text{ if }\mathbf{u}'_1 + \cdots+\mathbf{u}'_{\ell} = \mathbf{0}, \\
\ell - (a_1+\cdots+a_s) &\text{ otherwise}. 
\end{cases}
\end{equation}
Batyrev~\cite{Batyrev} provided a criterion for a projective toric variety to be Fano or weak Fano.
\begin{proposition}[{\cite[Proposition~2.3.6]{Batyrev}}]\label{prop:batyrev}
A smooth compact toric  variety $X$ is Fano \textup{(}respectively, weak Fano\textup{)} if and only if $\mathrm{deg}(R)>0$ \textup{(}respectively, $\mathrm{deg}(R)\geq 0$\textup{)} for every primitive collection $R$ of the fan $\Sigma$ of $X$. 
\end{proposition}

A smooth compact toric variety is called a \emph{Bott manifold} if it is isomorphic to the total space of a Bott tower that is an iterated $\C P^1$-bundle starting with a point, where each $\C P^1$-bundle is the projectivization of Whitney sum of two complex line bundles.  It is known that a smooth compact toric variety $X$ is a Bott manifold if and only if the fan $\Sigma$ of $X$ has primitive collections $\PC(\Sigma)$ of the form in~\eqref{equation_PC_of_Bott}.  Theorem~\ref{thm_char_matrix_of_toric_Schubert} says that a toric Schubert variety is a Bott manifold. We refer the reader to~\cite{GK94Bott} for Bott towers, and to~\cite{MasudaPanov08} for details on Bott manifolds. 
\section{Toric Schubert varieties and directed graphs}\label{section_toric_Schubert_directed_graphs}

In this section, we associate an edge-labeled digraph $\G_w$ with a toric Schubert variety $X_w$ and prove that two toric Schubert varieties are isomorphic as varieties if and only if the associated edge-labeled digraphs are isomorphic (Theorem~\ref{theo:graph_determines_Xw}).  We also give a simple criterion of when $X_w$ is Fano or weak Fano in terms of $\G_w$.  

As before, let $W$ be the Weyl group of a simple Lie group $G$ of rank $r$ and $s_1,\dots,s_r$ the simple reflections in $W$. Suppose that a Schubert variety $X_w$ $(w\in W)$ is toric.  Then simple reflections in a reduced decomposition $w=s_{i_1}\cdots s_{i_m}$ for $w$ are mutually distinct by Theorem~\ref{thm_toric_Schubert_distinct}. 

\begin{definition}
For $w=s_{i_1}\cdots s_{i_m}$ above, we define an edge-labeled digraph $\G_{w}$ as follows.
\begin{itemize}
\item $V(\G_{w}) = \{i_1,\dots,i_m\}$; and
\item $(i_j,i_k) \in E(\G_{w})$ if and only if $c_{i_j,i_k} \neq 0$ for $1 \leq k < j \leq m$. 
\end{itemize}
Note that $1\le -c_{i_j, i_k}\le 3$ if $c_{i_j,i_k} \neq 0$.  We assign the positive integer $-c_{i_j, i_k}$ to the directed edge $(i_k,i_j) \in E(\G_{w})$.  When we draw $\G_w$, we omit the label $1$ for simplicity.   
\end{definition} 

Suppose that $c_{{i_k},i_{k+1}}=0$ for some $k\in [m-1]$, i.e. $s_{i_k}$ and $s_{i_{k+1}}$ commute.  Then  
replacing the factor $s_{i_k}s_{i_{k+1}}$ in the reduced decomposition $w=s_{i_1}\cdots s_{i_m}$ by $s_{i_{k+1}}s_k$ is called a {\it $2$-move}.  Word Property \cite[Theorem~3.3.1]{BB05Combinatorics} says that any two reduced decompositions for our $w$ are related via a sequence of $2$-moves.  This implies that the edge-labeled digraph $\G_w$ does not depend on the choice of the reduced decomposition for $w$.  

\begin{example} \label{exam:digraph_Gi}
\begin{enumerate}
\item \label{example_3.2_1} Let $G$ be of type $\dyn{A}$. 
For $w=s_2s_1s_3s_4=s_2s_3s_1s_4=s_2s_3s_4s_1$ and $w'=s_3s_4s_2s_1=s_3s_2s_4s_1=s_3s_2s_1s_4$, we have 
\begin{center}
$\mathcal G_{w} = $
\begin{tikzpicture}[node/.style={circle,draw, fill=white!20, inner sep = 0.25mm}, baseline = -0.5ex]
\node[node] (1) at (1,0) {$1$};
\node[node] (2) at (2,0) {$2$};
\node[node] (3) at (3,0) {$3$};
\node[node] (4) at (4,0) {$4$};
\draw[->] (1) to (2);
\draw[<-] (2) to (3);
\draw[<-] (3) to (4);
\end{tikzpicture}
$\quad \quad$
$\mathcal G_{w'} = $
\begin{tikzpicture}[node/.style={circle,draw, fill=white!20, inner sep = 0.25mm}, baseline = -0.5ex]
\node[node] (1) at (1,0) {$1$};
\node[node] (2) at (2,0) {$2$};
\node[node] (3) at (3,0) {$3$};
\node[node] (4) at (4,0) {$4$};
\draw[->] (1) to (2);
\draw[->] (2) to (3);
\draw[<-] (3) to (4);
\end{tikzpicture}
\end{center}
For $w''=s_1s_2s_4s_5=s_1s_4s_2s_5=s_1s_4s_5s_2=s_4s_1s_2s_5=s_4s_1s_5s_2=s_4s_5s_1s_2$, $\mathcal G_{w''}$ is not connected as follows: 
\begin{center} 
$\mathcal G_{w''} = $
\begin{tikzpicture}[node/.style={circle,draw, fill=white!20, inner sep = 0.25mm}, baseline = -0.5ex]
\node[node] (1) at (1,0) {$1$};
\node[node] (2) at (2,0) {$2$};
\node[node] (3) at (3,0) {$4$};
\node[node] (4) at (4,0) {$5$};
\draw[<-] (1) to (2);
\draw[<-] (3) to (4);
\end{tikzpicture}
\end{center}

\item \label{example_3.2_2} Let $G$ be of type $\dyn{C}_3$. For $w=s_1s_2s_3$ and $w' =s_3s_2s_1$, we have 
\[
\G_{w} = 
\begin{tikzpicture}[node/.style={circle,draw, fill=white!20, inner sep = 0.25mm}, baseline = -0.5ex]
\node[node] (1) at (1,0) {$1$};
\node[node] (2) at (2,0) {$2$};
\node[node] (3) at (3,0) {$3$};
\draw[<-] (1) to (2);
\draw[<-] (2) to node[above, midway] {$2$}  (3) ;
\end{tikzpicture}
\quad \quad 
\G_{w'} = 
\begin{tikzpicture}[node/.style={circle,draw, fill=white!20, inner sep = 0.25mm}, baseline = -0.5ex]
\node[node] (1) at (1,0) {$1$};
\node[node] (2) at (2,0) {$2$};
\node[node] (3) at (3,0) {$3$};
\draw[->] (1) to (2);
\draw[->] (2) to (3) ;
\end{tikzpicture}
\]

\item Let $G$ be of type $\dyn{D}_4$. For $w=s_2s_1s_3s_4=s_2s_3s_1s_4=s_2s_3s_4s_1=s_2s_1s_4s_3=s_2s_4s_1s_3=s_2s_4s_3s_1$, we have 
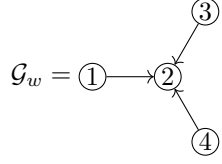
\[
\G_{w} = 
\begin{tikzpicture}
[node/.style={circle,draw, fill=white!20, inner sep = 0.25mm}, baseline = -0.5ex]
\node[node] (1) at (0,0) {$2$};
\node[node] (2) at (60:1) {$3$};
\node[node] (3) at (-60:1) {$4$};
\node[node] (4) at (180:1) {$1$};

\draw[<-] (1) to (2) ;
\draw[<-] (1) to (3) ;
\draw[<-] (1) to (4) ;
\end{tikzpicture} 
\]
\end{enumerate}
\end{example}

\begin{remark}
The underlying graph of $\G_w$ is an induced subgraph of the Dynkin diagram of $G$.
\end{remark} 

The edge-labeled digraph $\G_w$ has the same information as the primitive relations  \eqref{equation_ray_vectors2}.  Indeed, $\G_w$ can be obtained from the primitive relations  \eqref{equation_ray_vectors2} if we take a directed egde $(i_j,i_k)$ whenever $c_{i_j,i_k}$ in \eqref{equation_ray_vectors2} is non-zero and put the label $-c_{i_j,i_k}$ on it. Conversely, it is clear that the primitive relations \eqref{equation_ray_vectors2} can be obtained from $\G_w$ through this correspondence.  

We say that two edge-labeled digraphs are isomorphic if there is a bijection between their vertices preserving directed edges and labels on the edges.  For example, $\G_w$ and $\G_{w'}$ in Example~\ref{exam:digraph_Gi}\eqref{example_3.2_1} are isomorphic but those in Example~\ref{exam:digraph_Gi}\eqref{example_3.2_2} are not isomorphic as edge-labeled digraphs although they are isomorphic as digraphs.  

\begin{theorem}\label{theo:graph_determines_Xw}
Let $W$ and $W'$ be the Weyl groups of simple Lie groups $G$ and $G'$, respectively. Then toric Schubert varieties $X_w$ $(w\in W)$ and $X_{w'}$ $(w'\in W')$ are isomorphic as varieties if and only if $\G_{w}$ and $\G_{w'}$ are isomorphic as edge-labeled digraphs. 
\end{theorem}
\begin{proof}
First, we prove the \lq\lq only if\rq\rq\ part.  
Suppose that $X_w$ and $X_{w'}$ are isomorphic as varieties.  Then there is an isomorphism $\varphi$ from the fan $\Sigma$ of $X_w$ to the fan $\Sigma'$ of $X_{w'}$.  Let  
\begin{equation} \label{eq:ray_vectors_vw}
\text{$\{\mathbf v_{i_1},\dots,\mathbf v_{i_m}, \mathbf w_{i_1},\dots,\mathbf w_{i_m}\}$ and $\{\mathbf v'_{i'_1},\dots,\mathbf v'_{i'_m}, \mathbf w'_{i'_1},\dots,\mathbf w'_{i'_m}\}$}
\end{equation} 
be the ray generators (i.e., column vectors of the matrix in~\eqref{equation_ray_vectors}) of the fans $\Sigma$ and $\Sigma'$, respectively. Since the isomorphism $\varphi$ preserves the primitive collections,  there exists a bijection $f \colon [m] \to [m]$ such that 
\begin{equation}\label{eq_sum_vi_wi_and_image_f}
\{\varphi(\mathbf v_{i_k}), \varphi(\mathbf w_{i_k}) \}=\{\mathbf v'_{i'_{f(k)}}, \mathbf w'_{i'_{f(k)}}\}\qquad (\forall k\in [m]).
\end{equation}
Therefore, 
\begin{equation} \label{eq:sum_vw1}
\varphi(\mathbf v_{i_k} + \mathbf w_{i_k}) =\varphi(\mathbf v_{i_k})+\varphi(\mathbf w_{i_k})= \mathbf v'_{i'_{f(k)}} + \mathbf w'_{i'_{f(k)}}=\sum_{f(j)>f(k)}\left(-c_{i'_{f(j)},i'_{f(k)}}\right)\mathbf v'_{i'_{f(j)}}
\end{equation}
while sending the identity \eqref{equation_ray_vectors2} by $\varphi$, we obtain 
\begin{equation} \label{eq:sum_vw2}
\varphi(\mathbf v_{i_k} + \mathbf w_{i_k}) =\sum_{j > k} \left(-c_{i_j,i_k}\right) \varphi(\mathbf v_{i_j}).  
\end{equation}
It follows from \eqref{eq:sum_vw1} and \eqref{eq:sum_vw2} that 
\begin{equation} \label{eq:sum_vw3}
\sum_{f(j)>f(k)}\left(-c_{i'_{f(j)},i'_{f(k)}}\right)\mathbf v'_{i'_{f(j)}}
=\sum_{j > k} \left(-c_{i_j,i_k}\right) \varphi(\mathbf v_{i_j}).
\end{equation}
Here $\varphi(\mathbf v_{i_j})=\mathbf v'_{i'_{f(j)}}$ or $\mathbf w'_{i'_{f(j)}}$ by \eqref{eq_sum_vi_wi_and_image_f}.  When $c_{i_j,i_k}\not=0$, the latter case does not occur because if it occurs, the vector at the right hand side in \eqref{eq:sum_vw3} has a negative component coming from $\mathbf w'_{i'_{f(j)}}$ while all components of the vector at the left hand side in \eqref{eq:sum_vw3} are nonnegative, a contradiction. 
Thus \eqref{eq:sum_vw3} implies that 
\begin{equation} \label{eq:cij_f} 
c_{i_j,i_k} = c_{i_{f(j)}',i_{f(k)}'} \quad \text{ for all } 1 \leq k < j \leq m
\end{equation}
because $\mathbf v'_{i'_1},\dots,\mathbf v'_{i'_m}$ are linearly independent.   
Therefore, the bijection $f\colon [m]\to [m]$ provides an isomorphism from $\G_w$ to $\G_{w'}$ as edge-labeled digraphs.  

Now we prove the \lq\lq if\rq\rq\ part.  Suppose that $\G_w$ and $\G_{w'}$ are isomorphic as edge-labeled digraphs.  Then there is a bijection $f\colon [m]\to [m]$ which satisfies \eqref{eq:cij_f}.  We shall observe that the linear automorphism $\varphi$ of $\R^m$ defined by $\varphi(\mathbf{v}_{i_j})=\mathbf{v}'_{i'_{f(j)}}$ for $j\in [m]$ provides an isomporphism from the fan $\Sigma$ to the fan $\Sigma'$.  First we note 
\begin{equation} \label{eq:varphi(w)}
\varphi(\mathbf{w}_{i_k})=\mathbf{w}'_{i'_{f(k)}} \quad (\forall k\in [m]).
\end{equation}
Indeed, \eqref{eq:sum_vw3} holds by \eqref{eq:cij_f} and it follows from \eqref{equation_ray_vectors2} and \eqref{eq:sum_vw3} that we have    
\[
\begin{split}
\varphi(\mathbf{v}_{i_k})+\varphi(\mathbf{w}_{i_k})&=\varphi(\mathbf{v}_{i_k}+\mathbf{w}_{i_k})=\sum_{j > k} \left(-c_{i_j,i_k}\right) \varphi(\mathbf v_{i_j})\\
&=\sum_{f(j)>f(k)}\left(-c_{i'_{f(j)},i'_{f(k)}}\right)\mathbf v'_{i'_{f(j)}}=\mathbf v'_{i'_{f(k)}} + \mathbf w'_{i'_{f(k)}}. 
\end{split}
\]
Since $\varphi(\mathbf{v}_{i_k})=\mathbf{v}'_{i'_{f(k)}}$ by definition of $\varphi$, the identity above implies \eqref{eq:varphi(w)} as we desired.  Any maximal cone in $\Sigma$ is spanned by $\mathbf{u}_{i_1},\dots,\mathbf{u}_{i_m}$ where $\mathbf{u}$ denotes either $\mathbf{v}$ or $\mathbf{w}$, and the same is true for $\Sigma'$.  Therefore, $\varphi$ sends maximal cones in $\Sigma$ to those in $\Sigma'$ bijectively.  This means that $\varphi$ is an isomorphism from $\Sigma$ to $\Sigma'$.  Hence $X_w$ and $X_{w'}$ are isomorphic as varieties.   
\end{proof}

The edge-labeled digraph $\G_{w}$ encodes all the geometrical information of the toric Schubert variety $X_w$ by Theorem~\ref{theo:graph_determines_Xw}.  We shall give a simple criterion of when $X_w$ is (weak) Fano in terms of $\G_w$.  
For each vertex $i_k \in [m]$ of $\G_{w}$, we define the \emph{indegree} $\indeg(i_k)$ of $i_k$ to be the sum of labels of the edges going into $i_k$:
\begin{equation}\label{eq:outdegree}
\indeg(i_k) := \sum_{(i_j,i_k) \in E(\G_w)} (-c_{i_j,i_k})= \sum_{j>k} (-c_{i_j,i_k}). 
\end{equation}

\begin{theorem}\label{thm_Fano_weak_Fano}
A toric Schubert variety $X_w$ is Fano \textup{(}resp. weak Fano\textup{)} if and only if every vertex of $\G_w$ has indegree at most $1$ \textup{(}resp. $2$\textup{)}.  
\end{theorem}
\begin{proof}
Let $\{ \mathbf v_{i_1},\dots,\mathbf{v}_{i_m},\mathbf{w}_{i_1},\dots,\mathbf{w}_{i_m}\}$ be the column vectors of the matrix in~\eqref{equation_ray_vectors}.  Then the primitive collections of the fan of $X_w$ are $\{ \mathbf v_{i_k}, \mathbf w_{i_k} \}$ for $k=1,\dots,m$ and the primitive relations are 
\begin{equation*}\label{eq_sum_vk_wk}
\mathbf v_{i_k} + \mathbf w_{i_k} = \sum_{j > k} (-c_{i_j,i_k})\mathbf{v}_{i_j}. 
\end{equation*}
Therefore, it follows from \eqref{eq:degree} that the degree of the primitive collection $\{\mathbf{v}_{i_k},\mathbf{w}_{i_k}\}$ is given by
\begin{equation*}\label{eq_deg_of_vk_wk}
\deg(\{\mathbf v_{i_k}, \mathbf w_{i_k}\}) = 2 - \sum_{j>k}(-c_{i_j,i_k})=2-\indeg(i_k),
\end{equation*}
where the latter equality follows from~\eqref{eq:outdegree}.  Thus, the theorem follows from Proposition~\ref{prop:batyrev}.
\end{proof}

In type $\dyn{A}$, the edge-labeled digraph $\G_w$ is a union of directed path graphs whose all edges have label $1$. Therefore, we have   
\begin{corollary}
In Type $\dyn{A}$, any toric Schubert variety $X_w$ is weak Fano.  Moreover, it is Fano if and only if 
$\G_w$ has no sink other than leaves.  
\end{corollary}

\section{Enumeration of toric Schubert varieties} \label{sect:enumeration}
As is considered in Section~\ref{section_toric_Schubert_directed_graphs}, the graph $\G_{w}$ encodes fruitful geometric information of a toric Schubert variety~$X_w$. In this section, we enumerate the isomorphism classes of (Fano or weak Fano) toric Schubert varieties using $\G_w$ in Proposition~\ref{cor_enumeration_isom_Coxeter_Schubert}.

Let $G$ be a simple Lie group of rank $r$ with Cartan matrix $C = (c_{i,j})_{1 \leq i, j \leq r}$. A bijection $\theta \colon [r] \to [r]$ is called a \emph{Dynkin diagram automorphism} if 
\[
c_{\theta(i),\theta(j)} = c_{i,j}
\]
hold for all $i,j$. For each Lie type, we recall from \cite[\S12.2]{Humphreys78Lie} the group $\Gamma$ of Dynkin diagram automorphism in Table~\ref{table_D_diagram_auto}. 
When the generators of the group $\Gamma$ of Dynkin diagram automorphism can be expressed as elements in $W$, that is, when $ws_iw^{-1}$ is again a simple reflection for all $i \in [r]$, we display the corresponding elements in the third column of the table.
\begin{table}
\begin{tabular}{c|ll}
\toprule
Type & $\Gamma$ & $w$ \\
\midrule 
$\dyn{A}_{r}$ & $\Z_2$ ($r \geq 2$) & $w_0$\\
$\dyn{B}_r$, $\dyn{C}_r$ & $1$ \\ 
$\dyn{D}_r$ & $\begin{cases}
\mathfrak{S}_3 & (r = 4) \\
\Z_2 & (r > 4) 
\end{cases}$ 
	& $w_0$ if $r$ odd\\[1.4em]
$\dyn{E}_6$ & $\Z_2$ & $w_0$\\
$\dyn{E}_7$ & $1$\\
$\dyn{E}_8$ & $1$\\
$\dyn{F}_4$ & $1$\\
$\dyn{G}_2$ & $1$ \\
\bottomrule
\end{tabular}
\caption{Dynkin diagram automorphisms}\label{table_D_diagram_auto}
\end{table}

An element of $W$ is called a \emph{Coxeter element} if it can be written as a product of all simple reflections $s_1,\dots,s_{r}$. Let $\Cox_W$ denote the set of all Coxeter elements in $W$. Using Theorem~\ref{theo:graph_determines_Xw}, we obtain the following result. 
\begin{proposition}\label{thm:dynkin1}
Let $W$ be the Weyl group of a simple Lie group $G$.
Let $w, w' \in \Cox_W$.
The following statements are equivalent:
\begin{enumerate}
\item $X_w \cong X_{w'}$ as varieties.\label{statement_1}
\item $\mathcal G_{w} \cong \mathcal G_{w'}$ as edge-labeled digraphs. \label{statement_2}
\end{enumerate}
If $G$ is of type $\dyn{A}$, $\dyn{D}_r$ \textup{(}$r$ odd\textup{)}, or $\dyn{E}_6$, then the above statements are equivalent to the following:
\begin{enumerate}
	 \setcounter{enumi}{2}
\item $w' = w$ or $w' = w_0 w w_0$. \label{statement_3}
\end{enumerate}
\end{proposition}

\begin{remark}
The statement~\eqref{statement_3} implies the statement~\eqref{statement_1} in Proposition~\ref{thm:dynkin1} for any (not necessarily toric) Schubert varieties (see~\cite{richmond2021isomorphism}). 
Indeed, we consider the composition $\sigma$ of the two isomorphisms~$\tau,\rho$ of $\GL_n(\C)$ defined by
\[
\tau(A)=\! ^tA^{-1},\quad \rho(A)=w_0Aw_0.
\]
Since both $\tau$ and $\rho$ send the upper triangular Borel subgroup $B$ to the lower triangular Borel subgroup, their composition $\sigma$ preserves $B$ and hence induces a variety automorphism of $\GL_n(\C)/B$. Since $\tau$ fixes permutation matrices, we have $\sigma(BwB)=B\sigma(w)B=Bw_0ww_0B$, which implies $\sigma(X_w)=X_{w_0ww_0}$. 
\end{remark}

Recall that for a Coxeter element $w \in \Cox_W$, the underlying graph of $\G_w$ is the same as the Dynkin diagram of $W$. 
Accordingly, 
by Theorem~\ref{theo:graph_determines_Xw}, the number of isomorphism classes of toric Schubert varieties given by $\Cox_W$ is the same as that of orientations of Dynkin diagram of~$W$.

\begin{proposition}\label{cor_enumeration_isom_Coxeter_Schubert}
For each simple Lie type of $G$, the number of isomorphism classes of Fano or weak Fano toric Schubert varieties in $G/B$ given by Coxeter elements is displayed in Table~\ref{table_cardinality_iso_classes_Coxeter}. 
\begin{table}
\begin{tabular}{c|lll}
\toprule
Type & $\#\{ X_w \mid w \in \Cox_W \}/\sim$ & weak Fano & Fano\\
\midrule 
$\dyn{A}_r$ & $\displaystyle
\begin{cases}
2^{r-2} & \text{ when $r$ is even and $r\geq 2$}, \\
2^{r-2} + 2^{\frac{r-3}{2}} & \text{ when $r$ is odd and $r \geq 3$}
\end{cases}
$
& $\begin{cases}
2^{r-2} \\
2^{r-2} + 2^{\frac{r-3}{2}} 
\end{cases}$
&$\begin{cases}
\frac{r}{2} \\
\frac{r+1}{2}
\end{cases}$
\\[1.4em]
$\dyn{B}_r$ & $\displaystyle 
\begin{cases}
2^{r-1}&\text{ when $r=2 \text{ or }3$,}\\
7 \times 2^{r-4} & \text{ when $r$ is even and $r \geq 4$}, \\
7\times 2^{r-4} + 2^{\frac{r-5}{2}} & \text{ when $r$ is odd and $r \geq 5$}
\end{cases}$ 
&  $\begin{cases}
2^{r-1}\\
{7} \times 2^{r-4}   \\
{7}\times 2^{r-4} + 2^{\frac{r-5}{2}} 
\end{cases}$  & 
$1$ \\[2em]
$\dyn{C}_r$ & $\displaystyle 
\begin{cases}
4&\text{ when $r=3$},\\
7 \times 2^{r-4} & \text{ when $r$ is even and $r \geq 4$}, \\
7\times 2^{r-4} + 2^{\frac{r-5}{2}} & \text{ when $r$ is odd and $r \geq 5$}
\end{cases}$ &  $\displaystyle 
\begin{cases}
3\\
{5} \times 2^{r-4} &   \\
{5}\times 2^{r-4} + 2^{\frac{r-5}{2}} & 
\end{cases}$ & 
{$\begin{cases}
2\\
\frac{r}{2}&\\
\frac{r+1}{2} 
\end{cases}$} 
 \\[2em]
$\dyn{D}_r$ & $\displaystyle\begin{cases}
4 & \text{ if } r = 4, \\
3\cdot 2^{r-3}& \text{ otherwise}.
\end{cases}$
& $\begin{cases}
3 \\
5 \cdot 2^{r-4} 
\end{cases}$
& $\begin{cases}
2 \\
r-1
\end{cases}$
\\[1.2em]
$\dyn{E}_6$ & $20$
& $17$ & $4$
\\
$\dyn{E}_7$ & $2^6$ & $56$ & $7$\\
$\dyn{E}_8$ & $2^7$ & $112$ & $8$\\
$\dyn{F}_4$ & $2^3$ & $6$ & $2$\\
$\dyn{G}_2$ & $2$ & $1$ & $1$ \\
\bottomrule 
\end{tabular}
\caption{The cardinalities of the isomorphism classes of toric Schubert varieties given by Coxeter elements}\label{table_cardinality_iso_classes_Coxeter}
\end{table}
\end{proposition}
\begin{proof} We provide proof by analyzing directed Dynkin diagrams case-by-case. 
We will use the numbering of the vertices in Table~\ref{table_finite}.
We consider the simply laced cases first and then treat non-simply laced cases. 

\medskip  
\noindent \underline{\underline{Type $\dyn{A}$}}: For type $\dyn{A}$, since there are at most two edges at each vertex, any toric Schubert variety is weak Fano by Theorem~\ref{thm_Fano_weak_Fano}. 

We enumerate the number of isomorphism classes depending on the parity of the rank $r$. When $r$ is even, consider any orientation on the corresponding Dynkin diagram, which is a path graph having $r$ vertices. The involution on the Dynkin diagram does not preserve the orientation. 
Therefore, the involution on the Dynkin diagram provides an involution on the set $\{X_w \mid w \in \Cox_W\}$ of toric Schubert varieties, and this involution has no fixed element. 
Since there are $r-1$ edges, the number of isomorphism classes is $2^{r-1}/2 = 2^{r-2}$. 

When $r$ is odd, there exist orientations which are fixed by the involution on the Dynkin diagram.  
In fact, a directed path graph $\G$ on $[r]$ is fixed by the involution if and only if 
\[
(i,i+1) \in E(\G) \iff (r-i+1, r-i) \in E(\G) \quad \text{ for any } 1 \leq i \leq \frac{r+1}{2}-1. 
\] 
Therefore, the number of orientations fixed by the involution is $2^{\frac{r+1}{2}-1}$. 
This also provides that the number of orientations that are not fixed by the involution is $2^{r-1}-2^{\frac{r+1}{2}-1}$.
Accordingly, the number of isomorphism classes is 
\[
(2^{r-1}-2^{\frac{r+1}{2}-1})/2 + 2^{\frac{r+1}{2}-1}
= 2^{r-2} + 2^{\frac{r-1}{2}-1}.
\]

Now we enumerate the isomorphism classes of Fano toric Schubert varieties in $\{X_w \mid w \in \Cox_W\}$. By Theorem~\ref{thm_Fano_weak_Fano}, we consider orientations, where each vertex has at most one incoming edge. 
Since we are considering a path graph, any sink is a leaf of the graph. Consider the following possibilities. 
\begin{enumerate}
\item The sink set is $\{1,r\}$: there is a unique source and each of vertices $2,\dots,r-1$ can be a source. Considering the involution, there are $\frac{r+1}{2}-1$ orientations when $r$ is odd; $\frac{r}{2} -1$ orientations when $r$ is even. 
\item The sink set is $\{1\}$: the vertex $r$ should be a source, and there is a unique orientation having a source at $r$ and sink at $1$. Moreover, this is isomorphic to the orientation having a source at $1$ and a sink at $r$. 
\end{enumerate}
Therefore, we have  $\frac{r+1}{2}-1+1 = \frac{r+1}{2}$ orientations satisfying the Fano condition (up to isomorphisms) when $r$ is odd; $\frac{r}{2}-1+1 = \frac{r}{2}$ orientations satisfying the Fano condition (up to isomorphisms) when $r$ is even. This proves the claim.

\medskip  
\noindent \underline{\underline{Type $\dyn{D}$}}: For type $\dyn{D}_r$, when $r = 4$, we have the four directed Dynkin diagrams up to isomorphisms. See Table~\ref{table_directed_D4}.
\begin{table}
\begin{tabular}{ccccc}
$\G_w$ &
\begin{tikzpicture}[scale=0.7, baseline=-.5ex]
\tikzset{every node/.style={scale=0.7}}

\node[Dnode] (1) at (0,0) {};
\node[Dnode] (2) at (60:1) {};
\node[Dnode] (3) at (180:1) {};
\node[Dnode] (4) at (300:1) {};

\draw[<-] (1) to (2) ;
\draw[<-] (1) to (3) ;
\draw[<-] (1) to (4) ;
\end{tikzpicture} & 
\begin{tikzpicture}[scale=0.7, baseline=-.5ex]
\tikzset{every node/.style={scale=0.7}}

\node[Dnode] (1) at (0,0) {};
\node[Dnode] (2) at (60:1) {};
\node[Dnode] (3) at (180:1) {};
\node[Dnode] (4) at (300:1) {};

\draw[<-] (2) to (1) ;
\draw[<-] (1) to (3) ;
\draw[<-] (1) to (4) ;
\end{tikzpicture}  & 
\begin{tikzpicture}[scale=0.7, baseline=-.5ex]
\tikzset{every node/.style={scale=0.7}}

\node[Dnode] (1) at (0,0) {};
\node[Dnode] (2) at (60:1) {};
\node[Dnode] (3) at (180:1) {};
\node[Dnode] (4) at (300:1) {};

\draw[<-] (2) to (1) ;
\draw[<-] (3) to (1) ;
\draw[<-] (1) to (4) ;
\end{tikzpicture}  & 
\begin{tikzpicture}[scale=0.7, baseline=-.5ex]
\tikzset{every node/.style={scale=0.7}}

\node[Dnode] (1) at (0,0) {};
\node[Dnode] (2) at (60:1) {};
\node[Dnode] (3) at (180:1) {};
\node[Dnode] (4) at (300:1) {};

\draw[<-] (2) to (1) ;
\draw[<-] (3) to (1) ;
\draw[<-] (4) to (1) ;
\end{tikzpicture} \\[1.5em]
maximum of $\indeg(k)$ & $3$ & $2$ & $1$ & $0$
\end{tabular}
\caption{Directed Dynkin diagrams of type $\dyn{D}_4$}\label{table_directed_D4}
\end{table}
When $r > 4$, the Dynkin subgraph consisting of vertices $1,\dots,r-2$ provides $2^{r-3}$ choices, and the edge connecting nodes $r-2$ and $r-1$; and the edge connecting nodes $r-2$ and $r$ provides $3$ choices. Accordingly, we have $2^{r-3} \times 3$ directed Dynkin diagrams up to isomorphisms.

To enumerate Fano or weak Fano toric Schubert varieties, we consider directions on the Dynkin diagram of type $\dyn{D}_n$. By Theorem~\ref{thm_Fano_weak_Fano}, to enumerate Fano toric Schubert varieties, it is enough to consider orientations, where each vertex has at most one incoming edge. 
Accordingly, any sink is a leaf of the graph and there should be at least two sinks. Otherwise, the trivalent vertex has at least two incoming edges. 
Consider the following possibilities.
\begin{enumerate}
\item The sink set is $\{1, r\}$: there is only one orientation satisfying the Fano condition. In this case, the vertex $r-1$ is the unique source. 
\item The sink set is $\{1, r-1\}$: there is only one orientation satisfying the Fano condition. In this case, the vertex $r$ is the unique source.
\item The sink set is $\{r-1, r\}$: there is only one orientation satisfying the Fano condition. In this case, the vertex $1$ is the unique source. 
\item The sink set is $\{1, r-1, r\}$: there are ${r-3}$ orientations satisfying the Fano condition. Indeed, each of the vertices $2,\dots,r-2$ can be a source. 
\end{enumerate}
Since the orientations in (1) and (2) provide isomorphic digraphs, we have $1 + 1 + r-3 = r-1$ orientations satisfying the Fano condition (up to isomorphisms). 

Now we enumerate weak Fano toric Schubert varieties. There is only one trivalent vertex in the graph and because of Theorem~\ref{thm_Fano_weak_Fano}, it is enough to consider the orientations such that the indegree of the trivalent vertex $r-2$ is at most $2$. 
Considering the orientations on the emanating edges of the vertex $r-2$, there are five cases up to isomorphisms. 
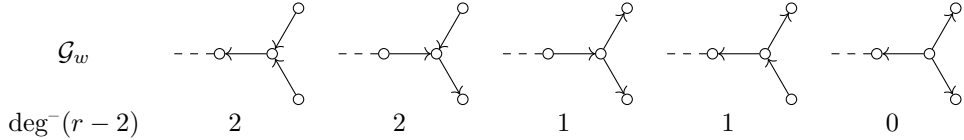
\begin{center}
\begin{tabular}{cccccc}
$\G_w$ &
\begin{tikzpicture}[scale=0.7, baseline=-.5ex]
\tikzset{every node/.style={scale=0.7}}

\node[Dnode] (1) at (0,0) {};
\node[Dnode] (2) at (60:1) {};
\node[Dnode] (3) at (-60:1) {};
\node[Dnode] (4) at (180:1) {};

\draw[dashed] (4) to (-2,0) ;
\draw[<-] (1) to (2) ;
\draw[<-] (1) to (3) ;
\draw[->] (1) to (4) ;
\end{tikzpicture} & 
\begin{tikzpicture}[scale=0.7, baseline=-.5ex]
\tikzset{every node/.style={scale=0.7}}

\node[Dnode] (1) at (0,0) {};
\node[Dnode] (2) at (60:1) {};
\node[Dnode] (3) at (-60:1) {};
\node[Dnode] (4) at (180:1) {};

\draw[dashed] (4) to (-2,0) ;
\draw[<-] (1) to (2) ;
\draw[->] (1) to (3) ;
\draw[<-] (1) to (4) ;
\end{tikzpicture}  & 
\begin{tikzpicture}[scale=0.7, baseline=-.5ex]
\tikzset{every node/.style={scale=0.7}}

\node[Dnode] (1) at (0,0) {};
\node[Dnode] (2) at (60:1) {};
\node[Dnode] (3) at (-60:1) {};
\node[Dnode] (4) at (180:1) {};

\draw[dashed] (4) to (-2,0) ;
\draw[->] (1) to (2) ;
\draw[->] (1) to (3) ;
\draw[<-] (1) to (4) ;
\end{tikzpicture} & 
\begin{tikzpicture}[scale=0.7, baseline=-.5ex]
\tikzset{every node/.style={scale=0.7}}

\node[Dnode] (1) at (0,0) {};
\node[Dnode] (2) at (60:1) {};
\node[Dnode] (3) at (-60:1) {};
\node[Dnode] (4) at (180:1) {};

\draw[dashed] (4) to (-2,0) ;
\draw[->] (1) to (2) ;
\draw[<-] (1) to (3) ;
\draw[->] (1) to (4) ;
\end{tikzpicture}&
\begin{tikzpicture}[scale=0.7, baseline=-.5ex]
\tikzset{every node/.style={scale=0.7}}

\node[Dnode] (1) at (0,0) {};
\node[Dnode] (2) at (60:1) {};
\node[Dnode] (3) at (-60:1) {};
\node[Dnode] (4) at (180:1) {};

\draw[dashed] (4) to (-2,0) ;
\draw[->] (1) to (2) ;
\draw[->] (1) to (3) ;
\draw[->] (1) to (4) ;
\end{tikzpicture}  \\[1.5em]
$\indeg(r-2)$ & $2$ & $2$ & $1$ & $1$ & $0$
\end{tabular}
\end{center}
Since there is no condition on the ${r-4}$ edges connecting vertices $1,\dots,r-3$, there are $5 \cdot 2^{r-4}$ orientations satisfying weak Fano condition. 

\medskip  

\noindent \underline{\underline{Type $\dyn{E}$}}: For type $\dyn{E}_6$, the longest element $w_0$ provides a Dynkin diagram involution. Since the Dynkin diagram of $\dyn{E}_6$ can be constructed by adding one more edge to the middle vertex in the Dynkin diagram of type $\dyn{A}_5$, the number of isomorphism classes is the same as twice of that of~$\dyn{A}_5$. Therefore, we obtain $2 \times 10 = 20$. 
For each type of $\dyn{E}_7$ and $\dyn{E}_8$, there does not exist a Dynkin diagram automorphism. Accordingly, there are $2^6$ (resp. $2^7$) isomorphism classes in type $\dyn{E}_7$ (resp.~$\dyn{E}_8$).

To enumerate (weak) Fano toric Schubert varieties in type $\dyn{E}_r$, we consider (weak) Fano in type $\dyn{A}_{r-1}$ first and then see how many choices of orientations there are for the remaining edge.  For instance, for type $\dyn{E}_6$, there are three Fano toric Schubert varieties in type $\dyn{A}_5$, where they are distinguished by the position of the source. Now we look at the orientation on the remaining edge of $\dyn{E}_6$.  If the source is the middle point, that is the vertex~$4$, then both orientations are fine but otherwise the orientation will be unique.  Therefore we obtain $2+1+1$ orientations satisfying Fano condition for Type $\dyn{E}_6$. A similar observation will work for $\dyn{E}_7$ and $\dyn{E}_8$. Indeed, for $r = 7,8$, an orientation satisfying the Fano condition in type $\dyn{A}_{r-1}$ has only one source. By considering the orientation on the remaining edge of $\dyn{E}_r$, if the source is the vertex~$4$, then both orientations are fine but otherwise, the orientation will be unique. Moreover, the Dynkin diagram automorphism on $\dyn{A}_{r-1}$ does not induce that on~$\dyn{E}_r$. Therefore, we have $r-2+2$ orientations satisfying the Fano condition for type $\dyn{E}_r$. Here, $r-2$ orientations are induced from that on~$\dyn{A}_{r-1}$ which do not have the source on the vertex $4$; and two orientations coming from that on $\dyn{A}_{r-1}$ which has the source on the vertex $4$.

Now we consider weak Fano toric Schubert varieties. 
For type $\dyn{E}_6$, we consider orientations in type $\dyn{A}_5$. There are $10$ different orientations in type $\dyn{A}_5$. We have the following three possibilities by looking at the middle vertex (which corresponds to the vertex $4$ in $\dyn{E}_6$. 
\begin{center}
\begin{tabular}{cccc}
$\G_w$ &
\begin{tikzpicture}[scale=0.7, baseline=-.5ex]
\tikzset{every node/.style={scale=0.7}}

\node[Dnode] (1) at (0,0) {};
\node[Dnode] (2) at (90:1) {};
\node[Dnode] (3) at (0:1) {};
\node[Dnode] (4) at (180:1) {};

\node[Dnode] (5) at (-2,0) {};
\node[Dnode] (6) at (2,0) {};

\draw[dashed] (4) to (5)
(3) to (6) 
(1) to (2);
\draw[->] (1) to (3) ;
\draw[<-] (1) to (4) ;
\end{tikzpicture} & 
 \begin{tikzpicture}[scale=0.7, baseline=-.5ex]
\tikzset{every node/.style={scale=0.7}}

\node[Dnode] (1) at (0,0) {};
\node[Dnode] (2) at (90:1) {};
\node[Dnode] (3) at (0:1) {};
\node[Dnode] (4) at (180:1) {};

\node[Dnode] (5) at (-2,0) {};
\node[Dnode] (6) at (2,0) {};

\draw[dashed] (4) to (5)
(3) to (6) 
(1) to (2);

\draw[<-] (1) to (3) ;
\draw[<-] (1) to (4) ;
\end{tikzpicture} &
\begin{tikzpicture}[scale=0.7, baseline=-.5ex]
\tikzset{every node/.style={scale=0.7}}

\node[Dnode] (1) at (0,0) {};
\node[Dnode] (2) at (90:1) {};
\node[Dnode] (3) at (0:1) {};
\node[Dnode] (4) at (180:1) {};

\node[Dnode] (5) at (-2,0) {};
\node[Dnode] (6) at (2,0) {};

\draw[dashed] (4) to (5)
(3) to (6) ;

\draw[dashed] (4) to (-2,0) 
(1) to (2);
\draw[->] (1) to (3) ;
\draw[->] (1) to (4) ;
\end{tikzpicture}   \\[1.5em]
\begin{tabular}{l}
$\#$ of orientations \\
(up to isomorphisms)
\end{tabular} & $4$ & $3$ & $3$ 
\end{tabular}
\end{center}
If the middle vertex is a sink, then we have only one orientation to be weak Fano but otherwise, both orientations are fine. Accordingly, we have $4 \times 2 + 3 \times 2 + 3 = 17$ orientations satisfying the weak Fano condition on type $\dyn{E}_6$ (up to isomorphisms). 

Now we consider $\dyn{E}_r$ for $r = 7,8$. Considering three edges emanating from the vertex $4$, there are seven orientations such that the vertex $4$ has at most $2$ indegree. Since there are no other conditions on the remaining $r-4$ edges, we have $7\times 2^{r-4}$ orientations satisfying the weak Fano condition on type $\dyn{E}_r$ for $r = 7,8$. 

\bigskip 

To enumerate the number of Fano or weak Fano toric Schubert varieties, we use the following observation. Suppose that the Dynkin diagram has a double or triple edge of the following form:
\[
\begin{tikzpicture}[scale =.5, baseline=-.5ex]
\tikzset{every node/.style={scale=0.7}}

\node[Dnode, label=below:{$i$}] (2) {};
\node[Dnode, label=below:{$j$}] (3) [right=of 2] {};

\draw[double line] (2)--(3);

\end{tikzpicture} \quad 
\begin{tikzpicture}[scale =.5, baseline=-.5ex]
\tikzset{every node/.style={scale=0.7}}

\node[Dnode,  label=below:{$i$}] (2) {};
\node[Dnode, label=below:{$j$}] (3) [right=of 2] {};

\draw[triple line] (2)--(3);
\draw (2)--(3);

\end{tikzpicture}
\]
From the definition of the edge labeling on the graph $\G_{w}$ and the definition of the Cartan integers, $s_j$ appears ahead of $s_i$ in the reduced word of $w$ if and only if the corresponding edge in the graph $\G_w$ has the label $-c_{i,j}$. 

\medskip  
\noindent \underline{\underline{Type $\dyn{B}$}}: 
We have the label $2$ on the edge $(r-1,r)$ on the graph $\G_w$ if $s_{r}$ appears ahead of $s_{r-1}$. In this case, we have $\indeg(r) = 2$:
\[
\begin{tikzpicture}
\node (1) at (1,0) {$r-1$}; 
\node[right = of 1] (2) {$r$};
\draw[->] (1) to node[above, midway] {$2$} (2);
\end{tikzpicture}
\]
Otherwise, we have label $1$ on the edge $(r,r-1)$, $\indeg(r)=0$ and $\indeg(r-1)\geq 1$. When $r=2$, there are two toric Schubert varieties which are weak Fano and one of which is Fano. When $r=3$, there are four toric Schubert varieties which are weak Fano and one of which are Fano. Now we assume $r\geq 4$.

We first enumerate the isomorphism classes of toric Schubert varieties. 
\begin{enumerate}
\item If we have the edge $(r-1,r)$, then there are $2^{r-2}$ different digraphs up to isomorphisms. 
\item If we have the edge $(r,r-1)$, then the vertex $r$ should be a source. We consider the following two cases separately. 
\begin{enumerate}
\item If the first vertex (with label 1 in the Dynkin diagram) is a sink, then $2^{r-3}$ graphs are all different. 
\item If the first vertex (with label 1 in the Dynkin diagram) is a source, then the number of isomorphism classes of the directed graphs is the same as that of the directed graphs of Lie type $\dyn{A}_{r-2}$. 
\end{enumerate}
\end{enumerate}
By adding all possible cases, we obtain that if $r$ is even and $r \geq 4$, then the number of isomorphism classes is  $7 \times 2^{r-4}$; 
If $r$ is odd and $r \geq 5$, then the number of isomorphism classes is $7\times 2^{r-4} + 2^{\frac{r-3}{2} - 1}$.

Moreover, we note that all orientations provide weak Fano toric Schubert varieties. To enumerate the isomorphism classes of Fano toric Schubert varieties, we consider the directed Dynkin diagrams which do not have label $2$, that is, we have the edge $(r,r-1)$ in $\G_w$. Since the graph~$\G_w$ has no sink other than leaves and the vertex $r$ is a source, the vertex $1$ should be a sink, which is the only possible orientation to be Fano. Indeed, the Coxeter element $s_1s_2 \cdots s_r$ provides a Fano toric Schubert variety. 

\medskip  
\noindent \underline{\underline{Type $\dyn{C}$}}: 
We have the label $2$ on the edge $(r,r-1)$ if $s_{r-1}$ appears ahead of $s_r$. 
Otherwise, we have the edge $(r-1,r)$ with label $1$. Therefore, we obtain the same number of the isomorphism classes in a similar way to the case of type $\dyn{B}_r$.  When $r=3$, there are four toric Schubert varieties three of which are weak Fano and two of which are Fano. Now we assume $r\geq 4$.

If the graph $\G_w$ has the edge $(r,r-1)$ with label $2$ and the edge $(r-2,r-1)$, then $r-1$ is a sink with $\indeg(r-1) = 3$: 
\[
\begin{tikzpicture}
\node (1) at (1,0) {$r-2$};
\node[right = of 1] (2)  {$r-1$};
\node[right = of 2] (3) {$r$};
\draw[<-] (2) to (1);
\draw[<-] (2) to node[above, midway] {$2$}  (3) ;
\end{tikzpicture}
\] 
Therefore, any orientation that extends the above orientation does not provide weak Fano. The number of such orientations is $2^{r-3}$. 
Since all the other orientations provide weak Fano,  we obtain the number of isomorphism classes of weak Fano toric Schubert varieties as in Table~\ref{table_cardinality_iso_classes_Coxeter}.

Recall that $X_w$ is Fano if and only if $\indeg(k) \leq 1$ for any vertex $k \in [r]$. Therefore, for $X_w$ to be Fano, there is no edge with label $2$ in the graph $\G_w$ and no sink on the vertices $\{2,\dots,r-1\}$. Accordingly, $r$ should be a sink and there exists only one source. Enumerating such orientations, we obtain $\frac{r+1}{2}$ orientations when $r$ is odd; $\frac{r}{2}$ orientations when $r$ is even in a similar way to the case of type $\dyn{A}_r$.

\medskip 
\noindent \underline{\underline{Type $\dyn{F}_4$}}: There are eight directed Dynkin diagrams and we display the value $\max\{\indeg(k)\}$ below.
\begin{center}
\begin{tabular}{cccc}
\toprule 
$\G_w$ & \begin{tikzpicture}[baseline = -0.5ex]
\foreach \x in {1,...,4}{
\node(\x) at (\x,0) {$\x$};
}
\draw[<-] (1) to (2);
\draw[<-] (2) to (3);
\draw[<-] (3) to (4);
\end{tikzpicture}
& \begin{tikzpicture}[baseline = -0.5ex]
\foreach \x in {1,...,4}{
\node(\x) at (\x,0) {$\x$};
}
\draw[<-] (2) to (1);
\draw[<-] (2) to  (3);
\draw[<-] (3) to (4);
\end{tikzpicture}
& \begin{tikzpicture}[baseline = -0.5ex]
\foreach \x in {1,...,4}{
\node(\x) at (\x,0) {$\x$};
}
\draw[<-] (1) to (2);
\draw[<-] (2) to (3);
\draw[<-] (4) to (3);
\end{tikzpicture} \\[1.2em]
maximum of $\indeg(k)$ & 1 & 2 & 1  \\
\midrule 
$\G_w$ &  \begin{tikzpicture}[baseline = -0.5ex]
\foreach \x in {1,...,4}{
\node(\x) at (\x,0) {$\x$};
}
\draw[<-] (2) to (1);
\draw[<-] (2) to  (3);
\draw[<-] (4) to (3);
\end{tikzpicture} 
& \begin{tikzpicture}[baseline = -0.5ex]
\foreach \x in {1,...,4}{
\node(\x) at (\x,0) {$\x$};
}
\draw[<-] (1) to (2);
\draw[<-] (3) to node[above, midway] {$2$}  (2);
\draw[<-] (3) to (4);
\end{tikzpicture} 
& \begin{tikzpicture}[baseline = -0.5ex]
\foreach \x in {1,...,4}{
\node(\x) at (\x,0) {$\x$};
}
\draw[<-] (2) to (1);
\draw[<-] (3) to node[above, midway] {$2$}  (2);
\draw[<-] (3) to (4);
\end{tikzpicture}
\\[1.2em]
maximum of $\indeg(k)$ & 2 & 3 & 3  \\
\midrule 
$\G_w$ & \begin{tikzpicture}[baseline = -0.5ex]
\foreach \x in {1,...,4}{
\node(\x) at (\x,0) {$\x$};
}
\draw[<-] (1) to (2);
\draw[<-] (3) to node[above, midway] {$2$}  (2);
\draw[<-] (4) to (3);
\end{tikzpicture}
& \begin{tikzpicture}[baseline = -0.5ex]
\foreach \x in {1,...,4}{
\node(\x) at (\x,0) {$\x$};
}
\draw[<-] (2) to (1);
\draw[<-] (3) to node[above, midway] {$2$}  (2);
\draw[<-] (4) to (3);
\end{tikzpicture}
& \\[1.2em]
maximum of $\indeg(k)$ & 2 & 2 &   \\
\bottomrule 
\end{tabular}
\end{center}
Accordingly, there are $6$ weak Fano toric Schubert varieties, and $2$ Fano toric Schubert varieties (up to isomorphisms).

\medskip 
\noindent \underline{\underline{Type $\dyn{G}_2$}}: There are two directed Dynkin diagrams: 
\[
\begin{tikzpicture}[baseline = -0.5ex]
\foreach \x in {1,...,2}{
\node(\x) at (\x,0) {$\x$};
}
\draw[<-] (1) to node[above, midway] {$3$} (2);
\end{tikzpicture} \qquad 
\begin{tikzpicture}[baseline = -0.5ex]
\foreach \x in {1,...,2}{
\node(\x) at (\x,0) {$\x$};
}
\draw[<-] (2) to   (1);
\end{tikzpicture}
\]
One provides a Fano toric Schubert variety and the other is not weak Fano. This completes the proof.
\end{proof}

\section{Cohomology ring distinguish toric Schubert varieties for simply laced types}\label{section_cohomology_ring_ADE}

In this section, we consider the family of all toric Schubert varieties $X_w$ satisfying that all the edges in $\G_w$ are labeled by $1$, which includes every toric Schubert variety in $G/B$ for a simple algebraic group $G$ of simply laced type. We show that for each toric Schubert variety $X_w$ in the family above, the edge-labeled digraph $\G_w$ is recovered (up to isomorphism) from the cohomology ring $H^\ast(X_w;\Z)$.

For the sake of simplicity, we define the following two sets.
\begin{align*}
    \eset{w}{j}&=\{k\mid (k,j)\in E(\mathcal{G}_w)\},\\
    \peset{w}{j}&=\{k\mid (j,k)\in E(\mathcal{G}_w)\}.
\end{align*}
That is, $\eset{w}{j}$ corresponds to the set of inward edges to $j$, and $\peset{w}{j}$ corresponds to the set of outward edges from ${j}$ in the digraph $\G_w$.

\begin{lemma}\label{lem:cohomology}
Let $W$ be the Weyl group of a Lie group $G$. To $w\in \Cox_{W}$, if all the edges in $\G_w$ are labeled by $1$, then
the cohomology ring of the toric Schubert variety $X_w$ is
\begin{equation*}
H^*(X_w)=\Z[x_{1},\dots,x_{r}]\left/\left(x_{j}^2-x_{j}\sum_{k\in \peset{w}{j}} x_{k}\,\middle|\, j=1,\dots,{r}\right)\right.,
\end{equation*}
where $r$ is the rank of $G$.
\end{lemma}
\begin{proof}
Applying Theorem~\ref{thm_char_matrix_of_toric_Schubert} to Jurikewicz's theorem (see~\cite{Jurkiewicz}), we have
\begin{equation}\label{equation_Jurikewicz_thm}
H^*(X_w)=\Z[y_1,\dots,y_{r}]\left/\left(y_j^2 + y_j \left(\sum_{k<j}c_{i_j,i_k}y_k \right)\,\middle|\, j=1,\dots,{r}\right)\right.,
\end{equation}
where $c_{i_j, i_k}$ are Cartan integers. 
Since we assume that all the edges in $\G_w$ are labeled by $1$, $c_{i_j, i_k}$ is $0$ or $-1$, and $c_{i_j,i_k}=-1$ if and only if there is an edge between $i_j$ and $i_k$ in the Dynkin diagram of $G$ such that $\|\alpha_{i_k}\| > \|\alpha_{i_j}\|$.
Hence, by changing $y_j\mapsto x_{i_j}$ for each $j=1,\dots,{r}$, we get
\[
H^*(X_w)=\Z[x_{i_1},\dots,x_{i_{r}}]\left/\left(x_{i_j}^2-x_{i_j}\sum_{(i_j,i_k) \in E(\mathcal{G}_w)}x_{i_k}\,\middle|\, j=1,\dots,{r}\right)\right..
\] 
This proves the lemma.
\end{proof}
We demonstrate Lemma~\ref{lem:cohomology} in the following example. 
\begin{example}\label{example_comology_ring}
Suppose that $w= s_3s_1s_4s_5s_2 \in \Cox_{\dyn{A}_5}$. Then the digraph $\G_w$ is given as follows:
\begin{center}
\begin{tikzpicture}[node/.style={circle,draw, fill=white!20, inner sep = 0.25mm}, baseline = -0.5ex]
\foreach \x in {1,...,5}{
\node[node] (\x) at (\x,0) {$\x$};
}
\draw[<-] (1) to (2);
\draw[<-] (3) to (2);
\draw[<-] (3) to (4);
\draw[<-] (4) to (5);
\end{tikzpicture}
\end{center}
Using the Jurikewicz's theorem (see~\eqref{equation_Jurikewicz_thm}) and the computation in Example~\ref{example_char_matrix_s31452}, the cohomology ring $H^*(X_w)$ of the Schubert variety $X_w$ is the truncated polynomial ring $\Z[y_1,\dots,y_5]/\mathcal{I}$, where the ideal $\mathcal{I}$ is generated by 
\[
y_1^2, \quad y_2^2, \quad y_3^2 - y_3(y_1), \quad y_4^2 - y_4(y_3),\quad y_5^2 - y_5(y_1+y_2).
\]
By taking the change of bases $y_j \mapsto x_{i_j}$ for $j=1,\dots,5$, we obtain that the cohomology ring $H^*(X_w)$ is the truncated polynomial right $\Z[x_1,\dots,x_5]/\mathcal{I}'$, where the ideal $\mathcal I'$ is generated by 
\[
x_3^2, \quad x_1^2, \quad x_4^2 - x_4(x_3), \quad x_5^2 - x_5(x_4), \quad x_2^2 - x_2(x_3+x_1).
\] 
\end{example}
We notice that in Example~\ref{example_comology_ring}, $x_i^2 = 0$ in the cohomology ring if and only if $i = 1$ or $3$. On the other hand, the sinks of the digraph ${\G}_w$ are $1$ and $3$. We will see that this observation holds in general. 

An element $z\in H^2(X_w)$ is \emph{primitive} if $z$ cannot be divided by an integer greater than $1$ and is called  \emph{square zero} if $z^2=0$ in $H^\ast(X_w)$.
From the ring presentation in Lemma~\ref{lem:cohomology}, we see that $x_i$ is a square zero primitive element if the vertex $i$ is a sink of ${\G}_w$. There are more square zero primitive elements in $H^2(X_w)$.
For simplicity, we set 
\begin{equation} \label{eq:alpha}
\alpha_j:=\sum_{k\in \peset{w}{j}}x_k \quad\text{for $j=1,\dots,{r}$}.
\end{equation}
Hence if $j$ is a sink, then $\alpha_j=0$; otherwise it is the sum of at most three $x_k$'s.
\begin{lemma}\label{lemm:sve}
Assume that all the edges in $\G_w$ are labeled by $1$. A square zero primitive element in $H^2(X_w)$ is one of the following forms up to sign: 
\[
x_j\quad\text{ and } \quad 2x_{k}-x_j \text{ \textup{(}if  $\peset{w}{k}=\{j\}$\textup{)}},
\]
where $j$ is a sink of ${\G}_w$. 
\end{lemma} 
\begin{proof} We first notice that $x_j^2=0$ in the cohomology $H^*(X_w)$ if and only if $j$ is a sink of the digraph ${\G}_w$. 
By \cite[Corollary~2.1]{ch-ma12}, a square zero primitive element is of the following form up to sign:
\[
x_k-\frac{1}{2}\alpha_k \quad \text{or}\quad 2x_k-\alpha_k
\]
for some $k$ with $\alpha_k^2=0$, where the former case occurs when $\alpha_k$ is divisible by $2$ and the latter occurs otherwise. In our case, $\alpha_k$ is divided by $2$ only when $\alpha_k=0$. If $\alpha_k\neq 0$, then $\alpha_k^2=0$ only when $\alpha_{k}$ equals $x_j$ for a sink $j$ of $\G_w$, that is, $\peset{w}{k}=\{j\}$ for a sink $j$ of $\G_w$. This completes the lemma.
\end{proof}

Since the mod 2 reduction of the elements in Lemma~\ref{lemm:sve} is $x_j$, we can find all sources of ${\G}_w$ by looking at square zero primitive elements in $H^2(X_w)$. 

Following \cite[Section 6]{ch-ma-ou17}, we call an element $\alpha\in H^2(X_w;\Z_2)$ (possibly $\alpha=0$) an \emph{eigenelement} if there exists $x\in H^2(X_w;\Z_2)$ such that 
\[
x^2=\alpha x,\quad x\neq0,\quad\text{ and }\quad x\neq\alpha.
\]
Moreover, we call such $x$ an \emph{eigenvector} associated to $\alpha$. 
We define $E(\alpha)$ as the set of elements $x$ with $x^2=\alpha x$ ($x$ may be $0$ or $\alpha$) for an eigenelement $\alpha$, which is called the \emph{eigenspace} associated to $\alpha$. 
In fact, $E(\alpha)$ forms a vector space because we are working over $\Z_2$. We denote by $\bar{E}(\alpha)$ the quotient of $E(\alpha)$ by the one-dimensional subspace $\langle \alpha\rangle$ spanned by $\alpha$ and call $\bar{E}(\alpha)$ the \emph{reduced eigenspace} associated with $\alpha$. 
Note that eigenelements, eigenspaces, and reduced eigenspaces are preserved under a cohomology ring isomorphism. 
For an eigenelement $\alpha \in H^2(X_w;\Z_2)$, we define its \emph{multiplicity} to be the dimension of $\bar{E}(\alpha)$.

\begin{lemma}[Lemma 6.2 in \cite{ch-ma-ou17}] \label{lem:eigenelements}
The eigenelements are $\alpha_j$'s \textup{(}regarded as elements over $\Z_2$\textup{)} and the reduced eigenspace $\bar{E}(\alpha)$ is spanned by $x_j$'s for each eigenelement $\alpha$.
\end{lemma}

We see an example for nonzero eigenelements and their eigenspaces.

\begin{example}\label{example1}
    Let $G$ be of type $\dyn{D}$ and $w=s_8s_5 s_4 s_2s_1s_3 s_6s_7s_9$. Then we have: 
    \begin{center}
    $\G_w=$
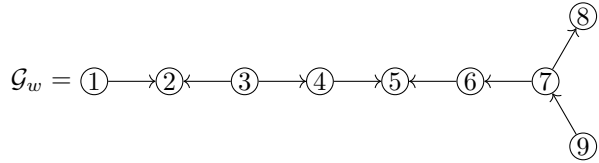
        \begin{tikzpicture}
            [node/.style={circle,draw, fill=white!20, inner sep = 0.25mm}, baseline = -0.5ex]
            \node[node] (7) at (0,0) {$7$};
            \node[node] (8) at (60:1) {$8$};
            \node[node] (9) at (-60:1) {$9$};
            \node[node] (6) at (180:1) {$6$};
            \node[node] (5) at (-2,0) {$5$};
            \node[node] (4) at (-3,0) {$4$};
            \node[node] (3) at (-4,0) {$3$};
            \node[node] (2) at (-5,0) {$2$};
            \node[node] (1) at (-6,0) {$1$};
            \draw[<-] (5) to (6);
            \draw[<-] (6) to (7);
            \draw[<-] (8) to (7);
            \draw[<-] (7) to (9);
            \draw[<-] (4) to (3);
            \draw[<-] (2) to (3);
            \draw[<-] (2) to (1);
            \draw[<-] (5) to (4);
        \end{tikzpicture}
    \end{center}
    Then $\alpha_2=\alpha_5=\alpha_8=0$, $\alpha_1=x_2$, $\alpha_3=x_2+x_4$, $\alpha_4=x_5$, $\alpha_6=x_5$, $\alpha_7=x_6+x_8$, and $\alpha_9=x_7$.
    Hence for nonzero eigenelements of $H^2(X_w;\mathbb{Z}_2)$, their  eigenspaces are given as follows:
    \begin{center}
        \begin{tabular}{c||c|c}
            \toprule 
            $\alpha(\neq 0)$ & ${E}(\alpha)$  & $\dim\bar{E}(\alpha)$\\
            \midrule
            $x_2$ & $\langle x_1, x_2\rangle$  & $1$\\
            \hline
            $x_2+x_4$ & $\langle x_3,x_2+x_4\rangle$ & $1$\\
            \hline
            $x_5$ & $\langle x_4,x_5,x_6\rangle$  & $2$\\
            \hline
            $x_6+x_8$ & $\langle x_7,x_6+x_8\rangle$  & $1$\\
            \hline
            $x_7$ & $\langle x_7,x_9\rangle$  & $1$\\
            \bottomrule
        \end{tabular}
    \end{center}
\end{example}

\begin{lemma} \label{lemm:eigenspace}
Assume that all the edges in $\G_w$ are labeled by $1$.  Let $\alpha$ be a non-zero eigenelement of $H^2(X_w)$. Then the following statements hold: 
\begin{enumerate}
	\item $1\leq \dim \bar{E}(\alpha)\leq 3$;
	\item $\dim \bar{E}(\alpha)>1$ if and only if $\alpha=x_j$, $|\eset{w}{j}|>1$, and each $k\in \eset{w}{j}$ satisfies $\peset{w}{k}=\{j\}$.
\end{enumerate}
\end{lemma}
\begin{proof}
By Lemma~\ref{lem:eigenelements} and the definition of $\alpha_j$, any non-zero eigenelement $\alpha$ is the element $\alpha_j=\sum_{k\in \peset{w}{j}} x_k$, where $j$ is not a sink. We prove the lemma according to the cardinality of $\peset{w}{j}$.

\smallskip 

\noindent \underline{Case 1}: Suppose that $x_j$ is an eigenelement. Then there exists $k\in \eset{w}{j}$ such that $\peset{w}{k}=\{j\}$, and 
\[
\bar{E}(x_j)=\langle x_k\mid k\in \eset{w}{j}\text{ with }\peset{w}{k}=\{j\}\rangle.
\] 
Therefore, the following hold:
\begin{enumerate}
    \item $1\leq \dim\bar{E}(x_j)\leq 3$.
    \item $\dim\bar{E}(x_j)=3$ if and only if $j$ is a sink at the trivalent vertex and each $k\in \eset{w}{j}$ is a source only when $k$ is a leaf.
    \item $\dim\bar{E}(x_j)=2$ if and only if $j$ satisfies one of the following:
    \begin{enumerate}
        \item $j$ is a sink and each $k\in \eset{w}{j}$ is a source only when $k$ is a leaf.
        \item $j$ is a trivalent vertex of $\G_w$ and there are two vertices $k\in\eset{w}{j}$ satisfying $\peset{w}{k}=\{j\}$.
    \end{enumerate}
\end{enumerate}

\smallskip 

\noindent \underline{Case 2}: Suppose that $\alpha_j=x_{i}+x_{k}$, where $i,j,k$ are pairwise distinct. If $\eset{w}{j}=\emptyset$, then $j$ is a source; otherwise, $j$ is a trivalent vertex of $\G_w$. In any case, {$E(\alpha_j) = \langle x_j, x_i+x_k\rangle$ and} $\bar{E}(\alpha_j)=\langle x_j\rangle$.

\smallskip 
\noindent \underline{Case 3}: Now we assume that $\alpha_j=x_i+x_k+x_\ell$, where $i,j,k,\ell$ are pairwise distinct. Then $j$ is a trivalent vertex and it is a source. Furthermore, {$E(\alpha_j) = \langle x_j, x_i+x_k+x_{\ell} \rangle$ and } $\bar{E}(\alpha_j)=\langle x_j\rangle$.
\end{proof}

Since the rank of $\overline{E}(\alpha)$ is the multiplicity of the eigenelement $\alpha$, the multiplicity of the zero eigenelement equals the number of sources in the graph.

\begin{theorem} \label{theo:recover}
The edge-labeled digraph $\mathcal{G}_w$ can be recovered {\rm(}up to isomorphism{\rm)} from the cohomology ring $H^*(X_w;\Z)$ if 
all the labels in $\G_w$ are $1$. 
\end{theorem}
Recall that for a toric Schubert variety $X_w$ in $G/B$ for a simple Lie group of type $\dyn{A}$, $\dyn{D}$, or $\dyn{E}$, every edge in $\G_w$ is labeled by~$1$. Hence we get the following result.
\begin{corollary}\label{thm:dynkin2}
Let $W$ and $W'$ be the Weyl groups of simple Lie groups $G$ and $G'$ of type $\dyn{A}$, $\dyn{D}$, or $\dyn{E}$. 
Let $w$ and $w'$ are elements in $W$ and $W'$, respectively.
The following statements are equivalent:
\begin{enumerate}
\item $H^{\ast}(X_w;\Z) \cong H^{\ast}(X_{w'};\Z)$ as graded rings.
\item $\G_{w} \cong \G_{w'}$ as digraphs.
\end{enumerate}
\end{corollary}
\begin{remark}
	Richmond and Slofstra~\cite{richmond2021isomorphism} studied a relation between isomorphism classes of (not necessarily toric) Schubert varieties and their cohomology rings. Indeed, they proved that two Schubert varieties are isomorphic (as algebraic varieties) if and only if there is a graded cohomology ring isomorphism preserving the Schubert bases. Here, for a toric Schubert variety of dimension~$m$, the cohomology classes $\{x_{i_1} \cdots x_{i_k} \mid 1 \leq i_1 < \cdots < i_k \leq m\}$ form the Schubert bases in terms of the cohomology ring presentation in Lemma~\ref{lem:cohomology}. 
\end{remark}

\begin{proof}[Proof of Theorem~\ref{theo:recover}]
We notice that the direction $(2) \implies (1)$ directly comes from Theorem~\ref{theo:graph_determines_Xw}. Hence it is enough to consider the direction $(1) \implies (2)$. Note that a toric Schubert variety $X_w$ is a product of toric Schubert varieties arising from Coxeter elements and a digraph $\G_w$ consists of the digraphs coming from Coxeter elements. Therefore, it suffices to prove the direction $(1) \implies (2)$ for Coxeter elements $w,w'\in\Cox_W$.
	
We recover the graph ${\G}_{w}$ from the cohomology ring $H^\ast(X_w)$ for a Coxeter element $w\in\Cox_W$. In Step~1, we define the $\Z_2$-vector spaces $V_1,\dots,V_r$ using eigenelements. 
We draw a graph whose vertices are labelled by $V_1,\dots,V_r$ from Step~2 to Step~5. In Step~2, we find a connected component starting from a one-dimensional space $V_j$. In Step~3, we draw a graph by combining the connected components obtained from Step~2. If there is no trivalent vertex in $\G_w$, 
we can recover the graph $\G_w$ at Step 3, but we may need Steps 4 and 5 for types $\dyn{D}$ and $\dyn{E}$. More precisely, we need to proceed with Step~4 if the trivalent vertex could not be recovered. We proceed with Step~5 if the graph obtained from Step~4 is disconnected.

\medskip
{\bf Step 1.} Note that there are exactly $r$ eigenelements considering multiplicities. For each vertex~$j$, we define a $\Z_2$-vector space $V_j$ by
\[
V_j=\begin{cases}
    \langle x_j\rangle &\text{ if }\alpha_j=0,\\
    E(\alpha_j)&\text{ otherwise.}
\end{cases}
\]
Note that $E(\alpha_{j})=E(\alpha_{j'})$ for different $j$ and $j'$ if the multiplicity of a nonzero eigenelement $\alpha_j$ is greater than $1$.

\medskip
{\bf Step 2.}
For each $j$ with $\alpha_j=0$, find all the spaces $V_a$'s satisfying $\dim (V_j\cap V_a) = 1$ and then draw a directed edge from $V_a$ to $V_j$. 
For a source $V_a$, we find all the spaces $V_b$'s satisfying $\dim (V_a\cap V_b) = 1$ and then draw a directed edge from $V_b$ to $V_a$. We repeat this process until there is no space $V_k$ satisfying that the dimension of the intersection of $V_k$ with a source of some connected component containing $V_j$ with $\alpha_j=0$ is one-dimensional. Then the resulting connected component is a path or a tree having one trivalent vertex.

If $\G_w$ has no trivalent vertex, then there are three possible forms by Lemma~\ref{lemm:eigenspace}.
\begin{equation} \label{eq:space-path1}
	\begin{tikzcd}[row sep = 0.5em ]
		\langle x_j\rangle  
		& E(\alpha_{j+1}) \lar \arrow[d, equal] 
		& \cdots \lar 
		& E(\alpha_{j+p}) \lar \arrow[d, equal]\\
		& \langle x_j, x_{j+1}\rangle 
		&& \langle x_{j+p-1},x_{j+p}\rangle 
	\end{tikzcd}
\end{equation}

\begin{equation} \label{eq:space-path2}
	\begin{tikzcd}[row sep = 0.5em] 
E(\alpha_{j-q}) \arrow[d, equal] \rar &  \cdots \rar 
& \rar E(\alpha_{j-1}) \arrow[d, equal]
&  \langle x_j\rangle \\
\langle x_{j-q},x_{j-q+1}\rangle 
&& \langle x_{j-1}, x_{j}\rangle
\end{tikzcd}
\end{equation}

\begin{equation} \label{eq:space-path3} 
E( \alpha_{j-q}) \rightarrow  E(\alpha_{j-q+1})\rightarrow \dots\rightarrow E(\alpha_{j-1})\rightarrow \langle x_j\rangle\leftarrow E(\alpha_{j+1})\leftarrow \dots\leftarrow E(\alpha_{j+p-1})\leftarrow  E(\alpha_{j+p}),
\end{equation}
where 
\begin{align*}
&E(\alpha_{j-1})=E(\alpha_{j+1})={E}(x_j)=\langle x_{j-1},x_j,x_{j+1}\rangle, \label{eq:both}\\
&{E}(\alpha_{j+t-1})=\langle x_{j+t-1},x_{j+t}\rangle\quad\text{ for $2\le t\le p$},\\
&{E}(\alpha_{j-s+1})=\langle x_{j-s},x_{j-s+1} \rangle\quad\text{ for $2\le s\le q$}.
\end{align*}
(Note that the multiplicity of the eigenelement $x_j$ is two and the last two identities above make sense when $p\ge 2$, $q\ge 2$.)

In types $\dyn{D}$ and $\dyn{E}$, we have four more forms {in each case}. For type $\dyn{D}$, we have the following four forms.

\begin{equation}\label{eq:space-tri1}
\begin{tikzpicture}[baseline=(current  bounding  box.east),
scale=.4,
terminal/.style={
    ellipse,
    minimum width=.8cm,
    minimum height=.5cm,
    font=\itshape,
},
]
\matrix[row sep=.1cm,column sep=.4cm] {%
     & &  & &   & & & \node [terminal](p1) {$E(\alpha_{r-1})$};\\
     \node [terminal](p6) {$E(\alpha_{r-q})$};& & \node [terminal](p5) {$\cdots$}; & &\node [terminal](p4) {$E(\alpha_{r-3})$}; &&\node [terminal](p3) {$\langle x_{r-2}\rangle$};&  \\
     & &  & & & & &\node [terminal](p2) {$E(\alpha_r)$};\\
};
\draw   (p1) edge [->,shorten <=2pt, shorten >=2pt] (p3)
        (p2) edge [->,shorten <=2pt, shorten >=2pt] (p3);
\draw   (p3) edge [<-, shorten <=2pt, shorten >=2pt] (p4)
        (p4) edge [<-, shorten <=2pt, shorten >=2pt] (p5)
        (p5) edge [<-, shorten <=2pt, shorten >=2pt] (p6); 
\end{tikzpicture}
\end{equation}
where the multiplicity of the eigenelement $x_{r-2}$ is three, $q\geq 3$, and  $E(\alpha_{r-3})=E(\alpha_{r-1})=E(\alpha_r)=E(x_{r-2})=\langle x_{r-3},x_{r-2},x_{r-1},x_r\rangle$.
\begin{equation}\label{eq:space-tri2}
\begin{tikzpicture}[baseline=(current  bounding  box.east),
scale=.4,
terminal/.style={
    ellipse,
    minimum width=.8cm,
    minimum height=.5cm,
    font=\itshape,
},
]
\matrix[row sep=.1cm,column sep=.4cm] {%
     & &  & &   & & & \node [terminal](p1) {$E(\alpha_{r-1})$};\\
     \node [terminal](p6) {$E(\alpha_{r-q})$};& & \node [terminal](p5) {$\cdots$}; & &\node [terminal](p4) {$E(\alpha_{r-3})$}; &&\node [terminal](p3) {$E(\alpha_{r-2})$};&  \\
     & &  & & & & &\node [terminal](p2) {$\langle x_r\rangle$};\\
};
\draw   (p1) edge [->, shorten <=2pt, shorten >=2pt] (p3)
        (p2) edge [<-, shorten <=2pt, shorten >=2pt] (p3);
\draw   (p3) edge [<-, shorten <=2pt, shorten >=2pt] (p4)
        (p4) edge [<-, shorten <=2pt, shorten >=2pt] (p5)
        (p5) edge [<-, shorten <=2pt, shorten >=2pt] (p6); 
\end{tikzpicture}
\end{equation}
where $q\geq 3$, $E(\alpha_{r-2})=E(x_r)=\langle x_{r-2},x_r\rangle$ and $E(\alpha_{r-3})=E(\alpha_{r-1})=E(x_{r-2})=\langle x_{r-3},x_{r-2},x_{r-1}\rangle$.

\begin{equation}\label{eq:space-tri3}
\begin{tikzpicture}[baseline=(current  bounding  box.east),
scale=.4,
terminal/.style={
    ellipse,
    minimum width=.8cm,
    minimum height=.5cm,
    font=\itshape,
},
]
\matrix[row sep=.1cm,column sep=.4cm] {%
     & &  & &   & & & \node [terminal](p1) {$\langle x_{r-1}\rangle$};\\
     \node [terminal](p6) {$E(\alpha_{r-q})$};& & \node [terminal](p5) {$\cdots$}; & &\node [terminal](p4) {$E(\alpha_{r-3})$}; &&\node [terminal](p3) {$E(\alpha_{r-2})$};&  \\
     & &  & & & & &\node [terminal](p2) {$E(\alpha_r)$};\\
};
\draw   (p1) edge [<-, shorten <=2pt, shorten >=2pt] (p3)
        (p2) edge [->, shorten <=2pt, shorten >=2pt] (p3);
\draw   (p3) edge [<-, shorten <=2pt, shorten >=2pt] (p4)
        (p4) edge [<-, shorten <=2pt, shorten >=2pt] (p5)
        (p5) edge [<-, shorten <=2pt, shorten >=2pt] (p6); 
\end{tikzpicture}
\end{equation}
where $q\geq 3$, $E(\alpha_{r-2})=E(x_{r-1})=\langle x_{r-2},x_{r-1}\rangle$ and $E(\alpha_{r-3})=E(\alpha_{r})=E(x_{r-2})=\langle x_{r-3},x_{r-2},x_{r}\rangle$.
\begin{equation}\label{eq:space-tri4}
\begin{tikzpicture}[baseline=(current  bounding  box.east),
scale=.4,
terminal/.style={
    ellipse,
    minimum width=.8cm,
    minimum height=.5cm,
    font=\itshape,
},
]
\matrix[row sep=.1cm,column sep=.4cm] {%
      & &   & & & \node [terminal](p1) {$E(\alpha_{r-1})$};\\
      \node [terminal](p5) {$\cdots$}; & &\node [terminal](p4) {$A$}; &&\node [terminal](p3) {$E(\alpha_{r-2})$};&  \\
      & & & & &\node [terminal](p2) {$E(\alpha_r)$};\\
};
\draw   (p1) edge [->, shorten <=2pt, shorten >=2pt] (p3)
        (p2) edge [->, shorten <=2pt, shorten >=2pt] (p3);
\draw   (p3) edge [->, shorten <=2pt, shorten >=2pt] (p4); 
\end{tikzpicture}
\end{equation}
where $A$ is $\langle x_{r-3}\rangle$ or $E(\alpha_{r-3})$, $E(\alpha_{r-2})=E(x_{r-3})=\langle x_{r-3},x_{r-2}\rangle$, and $E(\alpha_{r-1})=E(\alpha_{r})=E(x_{r-2})=\langle x_{r-2},x_{r-1},x_r\rangle$. Note that $A\cap E(\alpha_{r-2})=\langle x_{r-3}\rangle$ in any case.

In type $\dyn{E}$, we have four forms containing a trivalent vertex similarly to type $\dyn{D}$.

\begin{equation}\label{eq:space-e4}
\begin{tikzpicture}[baseline=(current  bounding  box.east),
scale=.4,
terminal/.style={
    ellipse,
    minimum width=.8cm,
    minimum height=.5cm,
    font=\itshape,
},
]
\matrix[row sep=.5cm,column sep=.4cm] {%
     &&& & \node[terminal](p2) {$E(\alpha_2)$}; & &   & & & &&\\
     \node [terminal](p1) {$E(\alpha_1)$};&&\node [terminal](p3) {$E(\alpha_3)$};& & \node [terminal](p4) {$\langle x_4\rangle$}; & &\node [terminal](p5) {$E(\alpha_5)$}; &&\node [terminal](pd) {$\cdots$};&  & \node [terminal](p6) {$E(\alpha_q)$};&\\
};
\draw   (p2) edge [->, shorten <=2pt, shorten >=2pt] (p4)
        (p1) edge [->, shorten <=2pt, shorten >=2pt] (p3)
        (p3) edge [->, shorten <=2pt, shorten >=2pt] (p4);
\draw   (p4) edge [<-, shorten <=2pt, shorten >=2pt] (p5)
        (p5) edge [<-, shorten <=2pt, shorten >=2pt] (pd)
        (pd) edge [<-, shorten <=2pt, shorten >=2pt] (p6);  
\end{tikzpicture}
\end{equation}
where $q\geq 5$.
\begin{equation}\label{eq:space-e5}
\begin{tikzpicture}[baseline=(current  bounding  box.east),
scale=.4,
terminal/.style={
    ellipse,
    minimum width=.8cm,
    minimum height=.5cm,
    font=\itshape,
},
]
\matrix[row sep=.5cm,column sep=.4cm] {%
     &&& & \node[terminal](p2) {$\langle x_2\rangle$}; & &   & & & &&\\
     \node [terminal](p1) {$E(\alpha_1)$};&&\node [terminal](p3) {$E(\alpha_3)$};& & \node [terminal](p4) {$E(\alpha_4)$}; & &\node [terminal](p5) {$E(\alpha_5)$}; &&\node [terminal](pd) {$\cdots$};&  & \node [terminal](p6) {$E(\alpha_q)$};&\\
};
\draw   (p2) edge [<-, shorten <=2pt, shorten >=2pt] (p4)
        (p1) edge [->, shorten <=2pt, shorten >=2pt] (p3)
        (p3) edge [->, shorten <=2pt, shorten >=2pt] (p4);
\draw   (p4) edge [<-, shorten <=2pt, shorten >=2pt] (p5)
        (p5) edge [<-, shorten <=2pt, shorten >=2pt] (pd)
        (pd) edge [<-, shorten <=2pt, shorten >=2pt] (p6);  
\end{tikzpicture}
\end{equation}
where $q\geq 5$.
\begin{equation}\label{eq:space-e6}
\begin{tikzpicture}[baseline=(current  bounding  box.east),
scale=.4,
terminal/.style={
    ellipse,
    minimum width=.8cm,
    minimum height=.5cm,
    font=\itshape,
},
]
\matrix[row sep=.5cm,column sep=.4cm] {%
     &&& & \node[terminal](p2) {$E(\alpha_2)$}; & &   & & & &&\\
     \node [terminal](p1) {$E(\alpha_1)$};&&\node [terminal](p3) {$\langle x_3\rangle$};& & \node [terminal](p4) {$E(\alpha_4)$}; & &\node [terminal](p5) {$E(\alpha_5)$}; &&\node [terminal](pd) {$\cdots$};&  & \node [terminal](p6) {$E(\alpha_q)$};&\\
};
\draw   (p2) edge [->, shorten <=2pt, shorten >=2pt] (p4)
        (p1) edge [->, shorten <=2pt, shorten >=2pt] (p3)
        (p3) edge [<-, shorten <=2pt, shorten >=2pt] (p4);
\draw   (p4) edge [<-, shorten <=2pt, shorten >=2pt] (p5)
        (p5) edge [<-, shorten <=2pt, shorten >=2pt] (pd)
        (pd) edge [<-, shorten <=2pt, shorten >=2pt] (p6);  
\end{tikzpicture}
\end{equation}
where $q\geq 5$.
\begin{equation}\label{eq:space-e7}
\begin{tikzpicture}[baseline=(current  bounding  box.east),
scale=.4,
terminal/.style={
    ellipse,
    minimum width=.8cm,
    minimum height=.5cm,
    font=\itshape,
},
]
\matrix[row sep=.5cm,column sep=.4cm] {%
     &&& & \node[terminal](p2) {$E(\alpha_2)$}; & &   & & & &&\\
     \node [terminal](p1) {$E(\alpha_p)$};&&\node [terminal](p3) {$\cdots$};& & \node [terminal](p4) {$E(\alpha_4)$}; & &\node [terminal](p5) {$\langle x_5\rangle$}; &&\node [terminal](pd) {$\cdots$};&  & \node [terminal](p6) {$E(\alpha_q)$};&\\
};
\draw   (p2) edge [->, shorten <=2pt, shorten >=2pt] (p4)
        (p1) edge [->, shorten <=2pt, shorten >=2pt] (p3)
        (p3) edge [->, shorten <=2pt, shorten >=2pt] (p4);
\draw   (p4) edge [->, shorten <=2pt, shorten >=2pt] (p5)
        (p5) edge [<-, shorten <=2pt, shorten >=2pt] (pd)
        (pd) edge [<-, shorten <=2pt, shorten >=2pt] (p6);  
\end{tikzpicture}
\end{equation}
where $q\geq 6$ and $p$ is $1$ or $3$. If the graph obtained from this step is a connected graph on the vertices $V_1,\dots,V_r$, then it is the desired graph. Otherwise, we need to proceed to Step~3.

\medskip

{\bf Step 3.} Let $\mathcal{R}$ be the set of subspaces $V_j$ not considered in Step 2, and let $\mathcal{L}$ be the set of leaves of connected components obtained from Step 2 as an undirected graph.

For each $V_j=E(\alpha_j)\in \mathcal{R}$, there are three possibilities.
\begin{enumerate}
    \item[(1)] There exist $V_p,V_q\in \mathcal{L}$ such that $\alpha_j\in V_p\oplus V_q$.
    \item[(2)] There exist $V_p,V_q,V_s\in \mathcal{L}$ such that $\alpha_j\in V_p\oplus V_q\oplus V_s$.
    \item[(3)] $V_j$ does not satisfy either of the above two.
\end{enumerate}
Let $\mathcal{R}_1$ (resp. $\mathcal{R}_2$ and $\mathcal{R}_3$) be the set of spaces $V_j$ satisfying (1) (resp. (2) and (3)) in the above. Then $\mathcal{R}$ is the disjoint union of $\mathcal{R}_1$, $\mathcal{R}_2$, and $\mathcal{R}_3$.

Now we use the spaces in $\mathcal{R}_1\cup \mathcal{R}_2$ to combine connected components obtained from Step 2. 

\begin{enumerate}[(i)]
\item  For $V_j=E(\alpha_j)\in \mathcal{R}_1$, we choose $V_p,V_q\in \mathcal{L}$ as minimal as possible such that $\alpha_j\in V_p\oplus V_q$. Then the choice of $V_p$ and $V_q$ are unique. Then we draw directed edges $V_j \to V_p$ and $V_j \to V_q$.

\item  For $V_j=E(\alpha_j)\in\mathcal{R}_2$, we choose $V_p,V_q,V_s\in\mathcal{L}$ as minimal as possible such that $\alpha_j\in V_p\oplus V_q\oplus V_s$. The we draw directed edges $V_j \to V_p$, $V_j \to V_q$, and $V_j \to V_s$.
\end{enumerate}

If the graph obtained from this step is a connected graph on the vertices $V_1,\dots,V_r$, then it is the desired graph. Otherwise, we need to proceed to Step~4. In fact, this happens only for type~$\dyn{D}$ or $\dyn{E}$, and the trivalent vertex is not recovered yet.

\medskip

{\bf Step 4.} Let $\mathcal{V}$ be the set of $V_j$'s contained in some connected component of the graph obtained from Step~3. Then there exists a unique pair of spaces $V_a\in \mathcal{R}_3$ and $V_b\in \mathcal{V}$ such that $V_a\cap V_b$ is one-dimensional. We draw a directed edge from $V_a$ to $V_b$. Note that $V_b$ becomes a trivalent vertex. For $V_a$, we find a remaining space $V_j$ in $\mathcal{R}_3$ such that $V_a\cap V_j$ is one-dimensional, and then draw a directed edge from $V_j$ to $V_a$. We repeat this process until there is no space s$V_k\in\mathcal{R}_3$ satisfying that the dimension of the intersection of $V_k$ with a space in the component containing $V_a$ is one-dimensional.

\medskip

If the graph obtained from this step is a connected graph on the vertices $V_1,\dots,V_r$, then it is the desired graph. Otherwise, we need to proceed to Step~5.

{\bf Step 5.} Let $\tilde{\G}$ be the graph obtained from Step~4. Then $\tilde{\G}$ has two connected components and there exists a unique element $V_k\in \mathcal{R}_3$ such that $V_k$ is not contained in $\tilde{\G}$. We choose $V_p$ and $V_q$ from $\tilde{\G}$ as minimal as possible such that $\alpha_k\in V_p\oplus V_q$. We draw directed edges from $V_k$ to $V_p$ and from $V_k$ to $V_q$.
\end{proof}

\begin{remark}
We cannot extend Theorem~\ref{theo:recover} to other Lie types by simply following the steps considered in its proof. For a simple Lie group $G$ of type~$\dyn{B}_2$, the cohomology ring $H^\ast(X_{s_2s_1})$ is isomorphic to $\Z[x,y]/\langle x^2,y^2\rangle$. Hence, there are two primitive square zero elements $x$ and $y$. This produces a graph consists of two vertices with no edges by following the steps in the proof of Theorem~\ref{theo:recover}. However, the digraph $\G_{s_2s_1}$ consists of two vertices and a directed edge with label~$2$.
\end{remark}

We demonstrate the steps in the proof of Theorem~\ref{thm:dynkin2} in the following examples. 
\begin{example}
    Let $G$ be of type $\dyn{D}$ and $w=s_8s_5 s_4 s_2s_1s_3 s_6s_7s_9$ as in Example~\ref{example1}. We recover $\G_w$ from $H^\ast(X_w)$ as follows.

    {\bf Step 1.} From Example~\ref{example1}, the spaces $V_j$'s are defined as follows.
    \begin{align*}
        &V_1=E(x_2)=\langle x_1,x_2\rangle,\,V_2=\langle x_2\rangle,\, V_3=E(x_2+x_4)=\langle x_2+x_4,x_3\rangle,\\ &V_4=E(x_5)=\langle x_4,x_5\rangle,\,V_5=\langle x_5\rangle,\,V_6=E(x_5)=\langle x_5,x_6\rangle, \\
       & V_7=E(x_6+x_8)=\langle x_6+x_8,x_7\rangle,\,V_8=\langle x_8\rangle,\,V_9=E(x_7)=\langle x_7,x_9\rangle
    \end{align*}

    {\bf Step 2.} We obtain the following connected components.
    \begin{align*}
        &V_1 \to V_2,\\
        &V_4 \to  V_5 \leftarrow V_6,\\
        &V_8
    \end{align*}

    {\bf Step 3.} We have $\mathcal{R}=\{V_3,V_7,V_9\}$ and $\mathcal{L}=\{V_1, V_2, V_4,V_6,V_8\}$. Since $x_2+x_4\in V_2\oplus V_4$ and $x_6+x_8\in V_6\oplus V_8$, we have $\mathcal{R}_1=\{V_3,V_7\}$, $\mathcal{R}_2=\emptyset$, and $\mathcal{R}_3=\{V_9\}$. Combining the connected components obtained from Step 2, we get the following graph.
    \begin{equation*}
\begin{tikzpicture}[baseline=(current  bounding  box.east),
scale=.4,
terminal/.style={
    ellipse,
    minimum width=.8cm,
    minimum height=.5cm,
    font=\itshape,
},
]
\matrix[row sep=.1cm,column sep=.4cm] {%
     & &  & &  && && && & & & \node [terminal](p8) {$V_8$};\\
     \node [terminal](p1) {$V_1$};& & \node [terminal](p2) {$V_2$}; & &\node [terminal](p3) {$V_3$}; &&\node [terminal](p4) {$V_4$};&  &\node [terminal](p5) {$V_5$};& &\node [terminal](p6) {$V_6$};& &\node [terminal](p7) {$V_7$};&\\
};
\draw   (p1) edge [->, shorten <=2pt, shorten >=2pt] (p2)
        (p2) edge [<-, shorten <=2pt, shorten >=2pt] (p3);
\draw   (p3) edge [->, shorten <=2pt, shorten >=2pt] (p4)
        (p4) edge [->, shorten <=2pt, shorten >=2pt] (p5)
        (p5) edge [->, shorten <=2pt, shorten >=2pt] (p6)
        (p6) edge [<-, shorten <=2pt, shorten >=2pt] (p7)
        (p8) edge [<-, shorten <=2pt, shorten >=2pt] (p7); 
\end{tikzpicture}
\end{equation*}

{\bf Step 4.} Since $V_7\cap V_9=\langle x_7\rangle$, we obtain the following graph, which is isomorphic to the graph~$\G_w$.
\begin{equation*}
\begin{tikzpicture}[baseline=(current  bounding  box.east),
scale=.4,
terminal/.style={
    ellipse,
    minimum width=.8cm,
    minimum height=.5cm,
    font=\itshape,
},
]
\matrix[row sep=.1cm,column sep=.4cm] {%
     & &  & &  && && && & & & \node [terminal](p8) {$V_8$};\\
     \node [terminal](p1) {$V_1$};& & \node [terminal](p2) {$V_2$}; & &\node [terminal](p3) {$V_3$}; &&\node [terminal](p4) {$V_4$};&  &\node [terminal](p5) {$V_5$};& &\node [terminal](p6) {$V_6$};& &\node [terminal](p7) {$V_7$};&\\
     & & && && && & & & & &\node [terminal](p9) {$V_9$};\\
};
\draw   (p1) edge [->, shorten <=2pt, shorten >=2pt] (p2)
        (p2) edge [<-, shorten <=2pt, shorten >=2pt] (p3);
\draw   (p3) edge [->, shorten <=2pt, shorten >=2pt] (p4)
        (p4) edge [->, shorten <=2pt, shorten >=2pt] (p5)
        (p5) edge [->, shorten <=2pt, shorten >=2pt] (p6)
        (p6) edge [<-, shorten <=2pt, shorten >=2pt] (p7)
        (p8) edge [<-, shorten <=2pt, shorten >=2pt] (p7)
        (p9) edge [->, shorten <=2pt, shorten >=2pt] (p7); 
\end{tikzpicture}
\end{equation*}
\end{example}

The following example shows why we need to take $V_p,V_q\in\mathcal{L}$ as minimal as possible in Step 3.

\begin{example}
Let $G$ be of type $\dyn{A}$ and $w=s_7s_8s_4s_5s_6s_2s_1s_3$. Then we have:
    \begin{center}
    $\G_w=$
        \begin{tikzpicture}
            [node/.style={circle,draw, fill=white!20, inner sep = 0.25mm}, baseline = -0.5ex]
            \node[node] (7) at (0,0) {$7$};
            \node[node] (8) at (1,0) {$8$};
            \node[node] (6) at (180:1) {$6$};
            \node[node] (5) at (-2,0) {$5$};
            \node[node] (4) at (-3,0) {$4$};
            \node[node] (3) at (-4,0) {$3$};
            \node[node] (2) at (-5,0) {$2$};
            \node[node] (1) at (-6,0) {$1$};
            \draw[<-] (5) to (6);
            \draw[->] (6) to (7);
            \draw[->] (8) to (7);
            \draw[<-] (4) to (3);
            \draw[<-] (2) to (3);
            \draw[<-] (2) to (1);
            \draw[->] (5) to (4);
        \end{tikzpicture}
    \end{center}
Note that $\alpha_2=\alpha_4=\alpha_7=0$ and $\alpha_1=x_2,\, \alpha_3=x_2+x_4,\, \alpha_5=x_4,\,\alpha_6=x_5+x_7,\,\alpha_8=x_7$. Now we recover $\G_w$ from $H^\ast(X_w)$.

{\bf Step 1.} The spaces $V_j$'s are defined as follows:
\begin{align*}
    &V_1=\langle x_1,x_2\rangle,\,V_2=\langle x_2\rangle,\,V_3=\langle x_3,x_2+x_4\rangle,\\
    &V_4=\langle x_4\rangle,\,V_5=\langle x_4,x_5\rangle,\,V_6=\langle x_6,x_5+x_7\rangle,\\
    &V_7=\langle x_7\rangle,\,V_8=\langle x_7,x_8\rangle
\end{align*}

{\bf Step 2.} We obtain the following connected components.
\begin{align*}
    &V_1 \rightarrow V_2,\\
    &V_4 \leftarrow V_5,\\
    &V_7 \leftarrow V_8.
\end{align*}

{\bf Step 3.} We have $\mathcal{R}=\{V_3,V_6\}$ and $\mathcal{L}=\{V_1,V_2,V_4,V_5,V_7,V_8\}$. Since $V_2\subset V_1$, $V_4\subset V_5$ and $V_7\subset V_8$, for each $V_j\in\mathcal{R}$ there are four possibilities of pairs $V_p$ and $V_q$ satisfying $\alpha_j\in V_p\oplus V_q$. If we choose $V_p$ and $V_q$ as minimal as possible, then $\alpha_3\in V_2\oplus V_4$ and $\alpha_6\in V_5\oplus V_7$, and we can recover the graph $\G_w$ as follows:
\[
V_1 \rightarrow V_2 \leftarrow V_3 \rightarrow V_4 \leftarrow V_5 \leftarrow  V_6 \rightarrow V_7 \leftarrow V_8
\]
\end{example}

In the following example, we need to proceed to Step~5 to recover $\G_w$.
\begin{example}
Let $G$ be of type $\dyn{E}_7$ and $w=s_7s_2s_1s_3s_4s_5s_6$. Then we have:
\begin{center}
    $\G_w=$
        \begin{tikzpicture}
            [node/.style={circle,draw, fill=white!20, inner sep = 0.25mm}, baseline = -0.5ex]
            \node[node] (7) at (0,0) {$7$};
            \node[node] (6) at (180:1) {$6$};
            \node[node] (5) at (-2,0) {$5$};
            \node[node] (4) at (-3,0) {$4$};
            \node[node] (3) at (-4,0) {$3$};
            \node[node] (2) at (-3,1) {$2$};
            \node[node] (1) at (-5,0) {$1$};
            \draw[<-] (5) to (6);
            \draw[->] (6) to (7);
            \draw[<-] (4) to (5);
            \draw[<-] (3) to (4);
            \draw[<-] (2) to (4);
            \draw[<-] (1) to (3);
        \end{tikzpicture}
    \end{center}
    Then $\alpha_1=\alpha_2=\alpha_7=0$ and $\alpha_3=x_1,\,\alpha_4=x_2+x_3,\,\alpha_5=x_4,\,\alpha_6=x_5+x_7$.

    {\bf Step 1.} The spaces $V_j$'s are defined as follows:
    \begin{align*}
        &V_1=\langle x_1\rangle,\, V_2=\langle x_2\rangle,\,V_3=\langle x_1,x_3\rangle,\,V_4=\langle x_2+x_3,x_4\rangle,\\
        &V_5=\langle x_4,x_5\rangle,\,V_6=\langle x_5+x_7,x_6\rangle,\,V_7=\langle x_7\rangle
    \end{align*}
    
    {\bf Step 2.} We obtain the following connected components.
    \[
    V_1 \leftarrow V_3,\quad V_2,\quad V_7
    \]

    {\bf Step 3.} We have $\mathcal{R}=\{V_4,V_5,V_6\}$ and $\mathcal{L}=\{V_1,V_2,V_3,V_7\}$. Since $\alpha_4\in V_2\oplus V_3$, we get $\mathcal{R}_1=\{V_4\}$, $\mathcal{R}_2=\emptyset$, and $\mathcal{R}_3=\{V_5,V_6\}$. In this step, we get the following graph.
    \begin{equation*}
\begin{tikzpicture}[baseline=(current  bounding  box.east),
scale=.4,
terminal/.style={
    ellipse,
    minimum width=.8cm,
    minimum height=.5cm,
    font=\itshape,
},
]
\matrix[row sep=.5cm,column sep=.4cm] {%
     &&& & \node[terminal](p2) {$V_2$}; & &   & & & &&\\
     \node [terminal](p1) {$V_1$};&&\node [terminal](p3) {$V_3$};& & \node [terminal](p4) {$V_4$}; & &\node [terminal](p5) {$~$}; &&\node [terminal](p6) {$~$};&  & \node [terminal](p7) {$V_7$};&\\
};
\draw   (p2) edge [<-, shorten <=2pt, shorten >=2pt] (p4)
        (p1) edge [<-, shorten <=2pt, shorten >=2pt] (p3)
        (p3) edge [<-, shorten <=2pt, shorten >=2pt] (p4); 
\end{tikzpicture}
\end{equation*}

{\bf Step 4} Since $V_4\cap V_5=\langle x_4\rangle$, we have $V_a=V_5$ and $V_b=V_4$. In this step, we get the following graph.
\begin{equation*}
\begin{tikzpicture}[baseline=(current  bounding  box.east),
scale=.4,
terminal/.style={
    ellipse,
    minimum width=.8cm,
    minimum height=.5cm,
    font=\itshape,
},
]
\matrix[row sep=.5cm,column sep=.4cm] {%
     &&& & \node[terminal](p2) {$V_2$}; & &   & & & &&\\
     \node [terminal](p1) {$V_1$};&&\node [terminal](p3) {$V_3$};& & \node [terminal](p4) {$V_4$}; & &\node [terminal](p5) {$V_5$}; &&\node [terminal](p6) {$~$};&  & \node [terminal](p7) {$V_7$};&\\
};
\draw   (p2) edge [<-, shorten <=2pt, shorten >=2pt] (p4)
        (p1) edge [<-, shorten <=2pt, shorten >=2pt] (p3)
        (p3) edge [<-, shorten <=2pt, shorten >=2pt] (p4)
        (p4) edge [<-, shorten <=2pt, shorten >=2pt] (p5); 
\end{tikzpicture}
\end{equation*}

{\bf Step 5.} Note that $\alpha_6\in V_5\oplus V_7$, and finally we can recover the graph $\G_w$ as follows.
\begin{equation*}
\begin{tikzpicture}[baseline=(current  bounding  box.east),
scale=.4,
terminal/.style={
    ellipse,
    minimum width=.8cm,
    minimum height=.5cm,
    font=\itshape,
},
]
\matrix[row sep=.5cm,column sep=.4cm] {%
     &&& & \node[terminal](p2) {$V_2$}; & &   & & & &&\\
     \node [terminal](p1) {$V_1$};&&\node [terminal](p3) {$V_3$};& & \node [terminal](p4) {$V_4$}; & &\node [terminal](p5) {$V_5$}; &&\node [terminal](p6) {$V_6$};&  & \node [terminal](p7) {$V_7$};&\\
};
\draw   (p2) edge [<-, shorten <=2pt, shorten >=2pt] (p4)
        (p1) edge [<-, shorten <=2pt, shorten >=2pt] (p3)
        (p3) edge [<-, shorten <=2pt, shorten >=2pt] (p4)
        (p4) edge [<-, shorten <=2pt, shorten >=2pt] (p5)
        (p5) edge [<-, shorten <=2pt, shorten >=2pt] (p6)
        (p7) edge [<-, shorten <=2pt, shorten >=2pt] (p6); 
\end{tikzpicture}
\end{equation*}
\end{example}


\end{document}